\documentclass[12pt]{amsart}

\usepackage{amsmath,amsfonts,amssymb,amsthm,mathtools,tikz-cd,tikz,xspace}
\usepackage{colonequals}
\usepackage{graphicx}
\usepackage[ruled,vlined,linesnumbered]{algorithm2e}

\usepackage[dvipsnames]{xcolor}
\definecolor{amethyst}{rgb}{0.6, 0.4, 0.8}

\usepackage[pagebackref=true, colorlinks, allcolors=amethyst]{hyperref}

\renewcommand*{\backref}[1]{}
\renewcommand*{\backrefalt}[4]{%
  \ifcase #1 %
    \relax
  \or
    $\uparrow$#2.%
  \else
    $\uparrow$#2.%
  \fi%
}

\usepackage{fullpage,url,amssymb,colonequals,algorithm2e}
\usepackage[shortlabels]{enumitem}
\usepackage{mathrsfs} 
\usepackage{tabularx} 
\usepackage{booktabs} 
\usepackage{pgfplots}
\pgfplotsset{compat=1.13}
\usetikzlibrary{arrows.meta,shapes.geometric, patterns}

\usepackage{subcaption}

\usepackage[T1]{fontenc}

\usepackage{cleveref}

\newcommand{\F}{\mathbb{F}}

\renewcommand{\P}{\mathbb{P}}
\newcommand{\Q}{\mathbb{Q}}

\newcommand{\Z}{\mathbb{Z}}

\newcommand{\calH}{\mathcal{H}}

\newcommand{\calO}{\mathcal{O}}

\DeclareMathOperator{\Aut}{Aut}

\DeclareMathOperator{\End}{End}

\DeclareMathOperator{\Gal}{Gal}

\DeclareMathOperator{\Hom}{Hom}

\DeclareMathOperator{\im}{im}

\DeclareMathOperator{\Pic}{Pic}

\DeclareMathOperator{\PSL}{PSL}

\DeclareMathOperator{\PGL}{PGL}

\DeclareMathOperator{\Sym}{Sym}

\DeclareMathOperator{\CM}{CM}

\theoremstyle{plain}
\theoremstyle{plain}
\theoremstyle{plain}

\counterwithin{equation}{section}

\newtheorem{lemma}[equation]{Lemma}
\newtheorem{proposition}[equation]{Proposition}
\newtheorem{corollary}[equation]{Corollary}
\newtheorem{conjecture}[equation]{Conjecture}
\theoremstyle{definition}
\newtheorem{definition}[equation]{Definition}
\newtheorem{example}[equation]{Example}
\theoremstyle{remark}
\newtheorem{remark}[equation]{Remark}

\theoremstyle{remark}

\newcommand{\LMFDBbase}{https://www.lmfdb.org}

\makeatletter
\DeclareRobustCommand{\ecl}[1]{\lmfdb@ecl#1\@nil}
\def\lmfdb@ecl#1.#2.#3\@nil{%
  \href{\LMFDBbase/EllipticCurve/Q/#1/#2/#3}{\texttt{#1.#2#3}}}
\makeatother

\makeatletter
\DeclareRobustCommand{\mf}[1]{\lmfdb@mf#1\@nil}
\def\lmfdb@mf#1.#2.#3.#4\@nil{%
  \href{\LMFDBbase/ModularForm/GL2/Q/holomorphic/#1/#2/#3/#4/}{\texttt{#1.#2.#3.#4}}}
\makeatother

\hypersetup{colorlinks=true, urlcolor=blue, citecolor=blue, linkcolor=blue}

\begin{document}
\title{Exceptional points on Atkin--Lehner quotients}
\author{Eran Assaf, Sachi Hashimoto, and Ari Shnidman}

\begin{abstract}
We study the rational points on the star curve $X_0^*(N) := X_0(N)/W(N)$, the quotient of the classical modular curve $X_0(N)$ by the full group of Atkin--Lehner involutions, for squarefree levels $N$. Rational points on $X_0^*(N)$ parameterize $\Q$-curves, i.e.\  elliptic curves $E/\overline{\Q}$ that are isogenous to all of their Galois conjugates. 
Elkies conjectures that $X_0^*(N)$ has only CM or cuspidal rational points for all large enough $N$.  We call any other rational points ``exceptional''. 
In this article, we provide new examples of exceptional points in genus 3 and 4, and we give evidence that no exceptional points exist in genus $g \geq 5$. 
Moreover, we investigate the underlying geometric reasons that might ``explain'' why these exceptional points arise in the first place, in the vein of Ogg and Mazur.
In particular, we propose geometric explanations for Galbraith's exceptional points on $X_0^*(137)$ and  $X_0^*(311)$.
\end{abstract}

\maketitle

\section{Introduction}

In his Eisenstein ideal paper, Mazur \cite{MazurEisenstein} showed that the modular curves $X_1(N)$ have non-cuspidal rational points if and only if the genus of $X_1(N)$ is $0$. In other words, the only rational points on the curves $X_1(N)$ are those that can be ``explained'' by geometry.  
Ogg pointed out that Mazur and Kenku's classification of the rational points on the curves $X_0(N)$ also has this feature: while there exist $X_0(N)$ of genus $g > 1$ with rational points that are neither CM nor cuspidal (henceforth {\it exceptional} rational points), these points nonetheless arise from cusps and CM points via geometry. 
For example, the two exceptional points on the genus $2$ curve $X_0(37)$ are images of cusps under the hyperelliptic involution. 

The curve $X_0(N)$ admits an action by the group of Atkin--Lehner involutions $W(N)$, and we consider the quotient curve $X_0^*(N) := X_0(N)/W(N)$, the \emph{star curve} of level $N$. Rational points on $X_0^*(N)$ give rise to low degree points on $X_0(N)$. They are also interesting in their own right, since they parametrize isogeny classes of \emph{$\Q$-curves}, elliptic curves $E/\overline{\Q}$ that are isogenous to all of their Galois conjugates.

As a generalization of the Serre Uniformity Conjecture to discrete arithmetic subgroups of $\PGL_2^+(\Q)$, Elkies \cite{ElkiesKCurves} conjectures that for any fixed $d \geq 1$, there are finitely many $N$ such that $X_0^*(N)$ has an exceptional point over a number field of  degree $d$. 
A main aim of this article is to give evidence for Elkies' conjecture and to formulate an effective conjecture over $\Q$.
In particular, we implement a broad search for exceptional points on squarefree levels in genus $3$ through $7$. Our search yields many new examples of $\Q$-curves in genus 3 and 4, including an example with $N = 510 = 2 \cdot 3 \cdot 5 \cdot 17$, but no examples in genus $g \geq 5$. 
To do this, we carefully apply Shimura reciprocity to provably and quickly enumerate all CM points of low degree on these curves. 

Our complementary aim is to determine to what extent the exceptional points that do exist ``arise from geometry'', in the vein of Ogg and Mazur. This question arose originally out of studying the intriguing examples of Galbraith \cite{Galbraithpplus}, who showed that the non-hyperelliptic genus $4$ curves $X_0^*(137)$ and $X_0^*(311)$ have exceptional points.

To explain these points, we look to recent work of Derickx--Hashimoto--Najman--Shnidman \cite{DHNS} who show that  all known rational points on all modular curves $X_G$, as $G$ varies over the open subgroups of $\mathrm{GL}_2(\widehat\Z)$, are explained by geometry in an appropriate sense. 
Building on these ideas, we show that all known exceptional points on star curves  $X_0^*(N)$ of squarefree level arise from cusps and CM points via three simple mechanisms that we call ``collinearity'', ``automorphism'', and ``elliptic covers''.  When feasible, we also check that these mechanisms do not give any other exceptional points. 

Combined with our broad search, this gives both empirical and theoretical evidence that we have found all exceptional points on all star curves $X_0^*(N)$ of squarefree level and that they all arise from CM and cuspidal points via geometry.

\subsection{Acknowledgements}

We thank Edgar Costa for invaluable computational assistance and resources. We thank Shiva Chidambaram for independently confirming Example \ref{ex: 286}, and we thank Abbey Bourdon and Filip Najman for helpful comments.
Assaf was supported by a grant from the Simons Foundation (SFI-MPS-Infrastructure-00008651, AS).
Hashimoto was partially supported by a grant from the National Science Foundation under award no.\ DMS-2501658.

\subsection{AI Statement}

We did not use AI tools to write this paper.
AI models (Claude Opus 5 and GPT-5.6 Sol) were used to produce code for the figures and tables in the paper with descriptive direction from the authors.
We used AI tools (Claude Opus 4.8, Opus 5 and Fable 5 (Thinking)) to review the paper, in particular to review the proofs of Section 3.3 on elliptic covers.
We also used AI tools in the development and review of the associated codebase.
In particular, Claude Fable, Opus 5, and Sonnet 5 were used to adapt initial author-written code for specific example curves, and generalize this code to work robustly across all levels, as well as to write tests and scripts.
We also used GPT-5.6 Sol and Claude to do code review.
We take full responsibility for the paper, the code and its output, and their mathematical correctness.

\section{Exceptional points on  Atkin--Lehner quotients of modular curves}
Let $N$ be a squarefree integer.  
We write $X_0^*(N)$ for the \emph{star curve} of level $N$, defined as the quotient of $X_0(N)$ by the full group of Atkin--Lehner automorphisms $W(N)$.
Elkies \cite{ElkiesKCurves} shows that, for any number field $K$, the non-cuspidal points in $X_0^*(N)(K)$ correspond to \emph{$K$-curves} (of degree $d \mid N$): elliptic curves $E/\overline{K}$ such that $E$ is $d$-isogenous to all its Galois conjugates.
In the case $K = \Q$, these $\Q$-curves are a mild generalization of elliptic curves defined over $\Q$; they arise as quotients of the curves $X_1(M)$ \cite{Ribet1992} \cite{KhareWintenberger}.

An important feature of star curves is that they tend to have many rational CM  points, corresponding to elliptic curves $E$ with complex multiplication over $\overline\Q$. Since $N$ is squarefree, $X_0^*(N)$ has exactly one (rational) cusp. We call cusps or CM points {\it special points}.  A non-special point in $X_0^*(N)(\Q)$ is called \emph{exceptional}.

Elkies conjectures that for $N \gg 1$, the points of $X_0^*(N)(\Q)$ are all special. In this section we give some evidence for a stronger boundedness conjecture.
\begin{conjecture}[Explicit boundedness of $\Q$-curves]
\label{conj:explicitbd}
Let $N>0$ be a squarefree integer. If the genus of $X_0^*(N)$ is at least 5, then the rational points are all special.
\end{conjecture}

We give additional evidence in the next section, where we give three geometric explanations for the existence of exceptional points in genus $g \leq 4$: collinearity, automorphism, and elliptic triple covers. In the case of collinearity, we give a heuristic for why this geometric phenomenon is unlikely to persist in higher genus. For automorphism and elliptic triple covers, we show that these explanations can only recover the known exceptional points.

For integers $q > 1$, there is a morphism $X_0^*(q^2 N) \to X_0^*(N)$ sending cusps to cusps, CM points to CM points, and exceptional points to exceptional points \cite{ElkiesKCurves}. 
This shows that 
every $\Q$-curve is isogenous to one of squarefree degree, which is one reason we focus on this case. 
Squarefree levels are also the most difficult to study: 
assuming the BSD conjecture, the rank of $J_0^*(N) := \mathrm{Jac} (X_0^*(N))$ is at least equal to the genus. 
In the non-squarefree case, one can often take advantage of the existence of rank 0 quotients \cite[Theorem 2]{HKLF} and use Mazur's formal immersion method.

\subsection{Method}

Our basic strategy to identify exceptional points proceeds as follows:
\begin{enumerate}
\item Construct a canonical model for $X_0^*(N)$ from modular forms $f_1, \dots, f_g \in S_2(\Gamma_0(N))$ that have Atkin--Lehner eigenvalue $+1$ for every $p \mid  N$.
\item Search for rational points in a box to find a lower bound on $X_0^*(N)(\Q)$.
\item Determine the exact number of rational CM points using Shimura reciprocity.
\item Compare: if we found more rational points than rational CM points plus 1 (for the rational cusp), then there is an exceptional point. 
\item Identify the CM points on the canonical model by evaluating $f_i(\tau)$.
\item Identify the cusp by taking a limit as $\tau \to i \infty$.
\item Find degree 2 CM points using Shimura reciprocity and identify them on the model.
\item ``Explain'' the exceptional point using the CM points and cusp.
\end{enumerate}
Steps (5) - (8) are strictly speaking not necessary for identifying exceptional points, but are used to ``explain'' them. We elaborate on what it means to ``explain'' exceptional points in Section \ref{sec:explanation}. One can find the code to carry out this method in \url{https://github.com/sachihashimoto/AtkinLehnerQuotients}; this repository can be used to access further information, including explicit equations for $X_0^*(N)$ and coordinates for each CM, cusp, and exceptional point. The exceptional-point tables are reproduced by
\path{tests/test_exceptional_tables.m} in that repository, and
Table~\ref{tab:all_planes} by
\path{scripts/make_plane_multiplicity_table.m}, which rebuilds every level
at \texttt{eval\_prec} $= 7000$.

There are only finitely many star curves of fixed genus $g$. For example, there are, respectively, $44,38,$ and $39$ star curves of squarefree level in genus $0, 1$, and $2$.

\begin{remark}
\label{rem:genus}
To list all levels $N$ such that $X_0^*(N)$ has genus $g$, we use the following.
All gonalities below are over $\overline{\Q}$. By Brill--Noether theory, any smooth projective genus $g$ curve has a map to $\P^1$ of degree at most $\lfloor \frac{g+3}{2} \rfloor$ \cite[Prop.~A.1(v)]{PoonenGonality}.  
Composing such a map with $X_0(N) \to X_0^*(N)$, we obtain the  upper bound $\gamma(X_0(N)) \leq 2^{\omega(N)}\lfloor \frac{g+3}{2} \rfloor$.
On the other hand, Abramovich \cite{Abramovich} gives the lower bound
$$
\gamma(X_0(N)) \geq \frac{\lambda_1}{24}[\PSL_2(\Z):\overline{\Gamma_0(N)}],
$$
where $\lambda_1$ is any lower bound for the smallest positive eigenvalue of the
Laplacian; the estimate $\lambda_1 \geq 975/4096$ of Kim--Sarnak
\cite[Appendix 2]{Kim} then yields
$\gamma(X_0(N)) \geq \frac{325}{2^{15}}[\PSL_2(\Z):\overline{\Gamma_0(N)}]$.
So
\[ \frac{[ \PSL_2(\Z):\overline{\Gamma_0(N)}] }{2^{\omega(N)}}
\leq \frac{2^{15}}{325}\left\lfloor \frac{g+3}{2} \right\rfloor \]
Combining this with a formula for the genus (see, e.g. \cite[Section 3]{GonzalezLario}) narrows down all candidate $N$ to the exact list.
\end{remark}

\subsection{Results}
We ran the above algorithm on the star curves of squarefree level for genus $2$ through $7$. (In genus $0$ and $1$, $X_0^*(N)$ has infinitely many rational points and  infinitely many exceptional points.)
This totals to 205 curves: 39 genus 2, 31 genus 3, 36 genus 4, 39 genus 5, 27 genus 6, and 33 genus 7.
The levels of $X_0^*(N)$ containing exceptional points for genus 3 and 4 are listed in Tables \ref{tab:exceptional_points_genus_3} and \ref{tab:exceptional_points_genus_4}. 
In genus 2, the numbers of exceptional points have been computed with elliptic and quadratic Chabauty calculations from the literature \cite{Qcurves,Genus2ANTS,AlgExamp}, and the CM and exceptional points were identified using $j$-invariants in \cite{Genus2ANTS}.

\begin{table}[ht]
\centering
\begin{tabular}{ccll}
\toprule
\textbf{Level}  & \textbf{Factorization} & \textbf{Collinearity} & \textbf{Automophism} \\
\midrule
$178$ & $2 \cdot 89$ &  $-388, -40$ & yes   \\
$183$ & $3 \cdot 61$ &  $-483$, $-75$ & yes \\
$246$ & $2 \cdot 3 \cdot 41$ &  $-264$, $-168$ & yes \\
$290$ & $2 \cdot 5 \cdot 29$ &  $-64$, $-24$ & yes\\
$310$ & $2 \cdot 5 \cdot 31$ &  $-55$, cusp  & no \\
$318$ & $2 \cdot 3 \cdot 53$ &  $-852$, $-372$ & yes  \\
$329$ & $7 \cdot 47$ &  $-52$, $-35$ & no \\
$430$ & $2 \cdot 5 \cdot 43$ &  $-220$, $-120$ & yes \\
$455$ & $5 \cdot 7 \cdot 13$ &  $-819$, $-195$ &yes \\
$510$ & $2 \cdot 3 \cdot 5 \cdot 17$ & $-480$, cusp  & yes  \\
\bottomrule
\end{tabular}
\caption{Exceptional points for genus $3$.  The Collinearity column records the
CM discriminants involved in a collinearity relation with the exceptional point; several
levels admit more than one, see Table~\ref{tab:all_planes}.}
\label{tab:exceptional_points_genus_3}
\end{table}

\begin{table}[ht]
\centering
\small
\begin{tabular}{clll}
\toprule
\textbf{Level} & \textbf{Factorization} & \textbf{Collinearity} & \textbf{Automorphism}   \\
\midrule
$137$ & $137$ & $-32$, $-11$, $-4$, cusp  & no\\
$311$ & $311$ & $-232$, $-123$, $-19$ & no \\
$370 $& $ 2 \cdot 5 \cdot 37$ & $-340$, $-260$, $-16$, cusp &yes \\
$399 $ & $ 3 \cdot 7 \cdot 19$ & $-1995$, $-483$, $-147$, $-84$  & no\\
\bottomrule
\end{tabular}
\caption{Exceptional points for genus $4$. The Collinearity column records the
CM discriminants involved in a collinearity relation with the exceptional point; see
Table~\ref{tab:all_planes} for the full list.}
\label{tab:exceptional_points_genus_4}
\end{table}

In genus 3 and 4, we find only one exceptional point per curve.
In genus 4, the exceptional points of level $N = 137$ and $311$ were discovered by Galbraith \cite{Galbraithpplus} who laid much of the computational groundwork for studying Atkin--Lehner quotients of modular curves and their CM points.  In his PhD Thesis, Galbraith \cite{GalbraithPhD} computed equations for Atkin--Lehner quotients of the modular curves $X_0(N)$ and identified CM points on these models by evaluation of $q$-expansions. 
We use and extend his method here.
The improvement we make over Galbraith's methods is to combine it with a general theoretical framework for determining all CM points on Atkin--Lehner quotients via Shimura reciprocity, allowing us to obtain a rigorous lower bound on the number of exceptional points.

In genus 5, 6, and 7, we do not find any exceptional points among the 99 curves. 
In particular, the following evidence supports the explicit boundedness conjecture (Conjecture~\ref{conj:explicitbd}).
\begin{proposition}
\label{prop:heightbd}
 Let $N$ be a squarefree integer such that $X_0^*(N)$ has genus $g\in \{5,6,7\}$. 
 Let $f_1,\dots,f_g$ be the basis, in Hermite normal form, of the saturated lattice of integral $q$-expansions of the weight-$2$ cusp forms of $\Gamma_0(N)$ fixed by every Atkin–Lehner involution, and let $X\subset\mathbb{P}^{g-1}$ be the associated canonical model of $X_0^*(N)$. 
 Then every rational point of $X$ of height at most $10^6$ is a special point. 
\end{proposition}

The coefficient sizes on these models are all quite small (0-3 digits, >75\% of non-zero coefficients are 1 digit) so it is reasonable to suspect that there are no more rational points.

\begin{remark}
It would be desirable to provably determine all exceptional points in low genus using the quadratic Chabauty method \cite{AlgExamp}. The main challenge is to compute the local heights at $N$.
\end{remark}

\begin{remark}
Significantly extending Proposition \ref{prop:heightbd}  seems infeasible without a large investment of computational resources, with searching for points of large height being the limiting factor. A point search up to height $10^5$ on the genus 8 curves found no exceptional points, and took 28 CPU hours on a MacBook Pro (Apple M4 Pro, 24 GB RAM).

But, as we discuss in Section \ref{sec:shimura}, computing the number of rational CM points on $X_0^*(N)$ and their discriminants is fast (about $0.2$s per curve in genus $7$). 
In this way, we can conjecturally compute $\#X_0^*(N)(\Q)$ for squarefree $N$; see Figure~\ref{fig:x0star-point-distributions}. 
\end{remark}

\begin{figure}[ht]
  \centering
\includegraphics[width = 400pt]{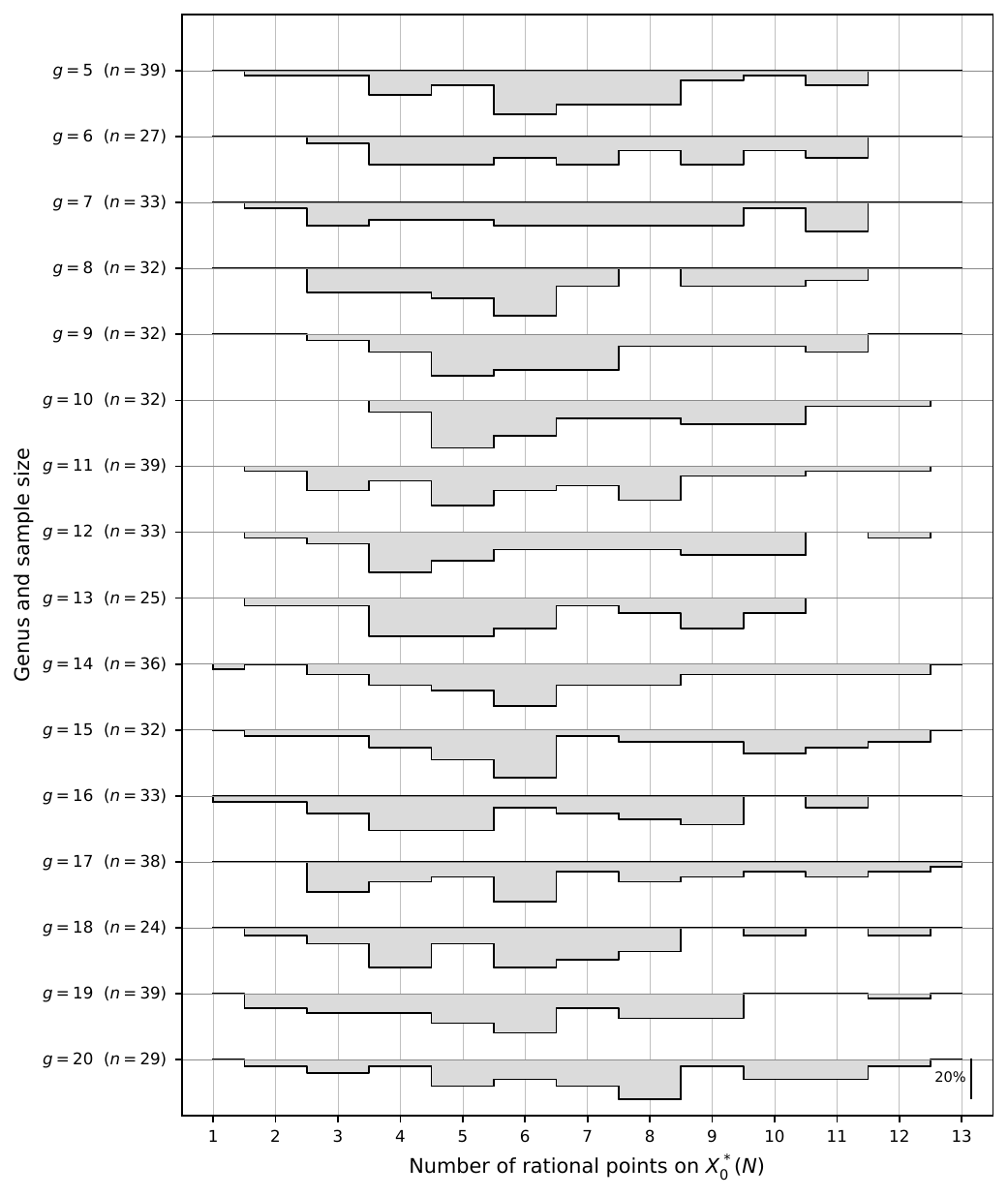}
  \caption{
  Distribution of the number of rational special points in
  \(X_0^*(N)(\Q)\) by genus, normalized by  \(n=\) number of curves. The vertical scale is the same across genera, and is labeled in the $g = 20$ row. Assuming Conjecture \ref{conj:explicitbd}, these numbers are also \(\#X_0^*(N)(\Q)\).
  }
  \label{fig:x0star-point-distributions}
\end{figure}

\subsection{Shimura reciprocity}

\label{sec:shimura}

To study the rationality properties of CM points on $X_0(N)$ and its Atkin--Lehner quotients we use \emph{Shimura reciprocity}. 
Let $R \subset \calO_K$ be an order of conductor $f$ in an imaginary quadratic field $K$. 
Let $\CM(R)$ be the set of points in $X_0(N)(\overline\Q)$ corresponding to cyclic $N$-isogenies $E \to E'$, where $E$ and $E'$ are elliptic curves with  CM by $\calO_E$ and $ \calO_{E'}$ respectively, such that $R = \calO_{E} \cap \calO_{E'}$.  The set $\CM(R)$ also corresponds bijectively to (conjugacy classes of) optimal embeddings of $R \hookrightarrow \calO $ into the Eichler order $\calO := (\begin{smallmatrix} \Z & \Z \\ N\Z & \Z \end{smallmatrix}) \subset \mathrm{Mat}_2(\Z)$ of level $N$. 
The set $\CM(R)$ is preserved by the action of $\Gal_\Q$ and Shimura reciprocity is an explicit description of this action in terms of optimal embeddings. 

Write $\overline{\CM}(R)$ for the image of $\CM(R)$ in $X_0^*(N)(\overline\Q)$ 
and define 
\[
N(R) \colonequals \prod_{p \mid  N, \left(\frac{K}{p}\right)=1 \text{ or } p \mid  f}  p, \quad
N^*(R) \colonequals \prod_{p \mid  N, \left(\frac{K}{p}\right)=1, p \nmid  f}  p
\]
In this subsection we prove the following proposition, and show how one can refine it to give the field of definition of the points in  $\overline{\CM}(R)$.
\begin{proposition}
\label{prop:degreeexact}
Assume $N$ is squarefree and $\CM(R) \ne \emptyset$.
The field of definition of any $P \in \overline{\CM}(R)$ has degree 
\begin{equation}
\label{eq:vareps}
[\Q(P) : \Q] = \frac{h(R)}{ 2^{\omega (N/N(R)) - \varepsilon(N,R) }},  \qquad 
\varepsilon(N,R)=
\begin{cases}
1 & \operatorname{disc}(R)\in \bigcup_{\substack{m\mid N, \\ m>1}} \Delta(m),
  \\ 0 & \text{otherwise,}
\end{cases} 
\end{equation}
where $\Delta(m)$ is the set of discriminants of orders in $\Q(\sqrt{-m})$ whose CM points are fixed by $w_m$ (see also \eqref{eq:delta_m}).
\end{proposition}

\begin{proof}
We first justify that the points in  $\CM(R)$   all have isomorphic fields of definition.
Let $H_R$ be the ring class field of $R$. 
Write $W(R) = \{ w_m \in W(N) : m \mid N(R)\}.$
The group
$W(R)\times \Gal(H_R/K)$ acts freely and transitively on $\CM(R)$ by \cite[Proposition~5.6]{GRgenus1}.
Thus for any  
$P, P' \in \CM(R)$, there exists $w \in W(R)$ and $\sigma \in \Gal(H_R/K)$ such that $P' = w(P^\sigma)$. Since $w$ is defined over $\Q$, it follows that $\Q(P') = \Q(P^\sigma) = \Q(P)^\sigma$. Hence $\Q(P') $ and $\Q(P)$ are isomorphic over $\Q$. Since the quotient map $X_0(N) \to X_0^*(N)$ is defined over $\Q$ and invariant under $W(R)$, this also holds for the points in $\overline{\CM}(R)$. 

Work of Ogg and Eichler (see Proposition~\ref{prop: oggeichler}) allows us to count the sizes of $\CM(R)$ and $\overline{\CM}(R)$, from which
Proposition~\ref{prop:degreeexact} follows.  By Proposition~\ref{prop:GR},
every $w_m$ with $m \mid N$ acts on $\CM(R)$ through $\Gal(H_R/K)$, or through
complex conjugation composed with such an element; since
$W(R) \times \Gal(H_R/K)$ acts freely and transitively on $\CM(R)$ by
\cite[Prop.~5.6]{GRgenus1}, it follows that $\Gal_\Q$ acts transitively on
$\overline{\CM}(R)$.  By orbit--stabilizer, the degree of the field of
definition of any point of $\overline{\CM}(R)$ is therefore
$\#\overline{\CM}(R)$, which is the asserted value by
\eqref{eq:number of CM points} and Proposition~\ref{prop: oggeichler}.
\end{proof}

\begin{remark} 
Explicitly, for each nontrivial divisor $m\mid N$
\begin{equation}
\label{eq:delta_m}
  \Delta(m)=
  \begin{cases}
  \{-4,-8\} & m=2,\\
  \{-m,-4m\} & m\equiv 3 \pmod 4,\\
  \{-4m\} & m\equiv 1,2 \pmod 4,\ m>2,
  \end{cases}
\end{equation}
\end{remark}

\begin{proposition}[Ogg, Eichler {\cite{Ogg}}]\label{prop: oggeichler}
Assume $N$ is squarefree.
The set $\CM(R)\neq \emptyset$ if and only if no prime $p\mid N$ is inert in $R$, in which case 
$$  \#\operatorname{CM}(R)=h(R)\cdot 2^{\omega(N(R))}, $$ where $h(R)=\#\operatorname{Pic}(R)$. 
On  
$X_0^*(N)$, 
\begin{equation}
\label{eq:number of CM points}
\#\overline{\operatorname{CM}}(R)=\frac{\#\operatorname{CM}(R)}{2^{\omega(N)-\varepsilon(N,R)}}.
\end{equation}
\end{proposition}

As a corollary of Proposition \ref{prop:degreeexact}, we see that for any fixed $N$ and $d$, there are finitely many orders $R$ such that $\overline{\CM}(R)$ contains degree $d$ points.
\begin{corollary}
\label{prop:bdclassno}
If $\overline{\CM}(R)$ gives points of degree $d$ on $X_0^*(N)$, then 
 $h(R)\leq d \cdot 2^{\omega(N)}.$
\end{corollary}

More generally, to compute the exact field of definition of a point in $\Q(\overline{\CM}(R))$, one can follow closely the methods of \cite[Appendix]{GRgenus1} which we have implemented for Atkin--Lehner quotients of Shimura curves in \cite{GitHub}. 
The most subtle point is to describe the action of the Atkin--Lehner involutions $w_m$ in terms of the $\Gal(H_R/K)$-action. We summarize this below.

\begin{proposition}[{Gonz\'alez, Rotger \cite[Lemma 5.9, Lemma 5.10]{GRgenus1}}]
\label{prop:GR}
If $m\mid N/N(R)$, then $w_m$ acts by an element $\sigma_m \in \Gal(H_R/K)$ on $\CM(R)$.
In addition, $w_{N^*(R)}$ acts by complex conjugation composed with an element $\sigma \in \Gal(H_R/K)$.
\end{proposition}

\begin{example}\label{ex: N = 286 D = -39}
Let $N = 286 = 2 \cdot 11 \cdot 13$ and let $R$ be the maximal order of $K = \Q(\sqrt{-39})$. 
We have $\Pic(R) \cong \Z/4 \Z$, so $h_R = 4$ and $[H_R: \Q] = 8$.
The primes $2$ and $11$ split in $K$ while $13$ ramifies, so $N(R) = 22 = N^*(R)$. 
Furthermore, $w_{13}$ acts by a Galois element in $\Gal(H_R/K)$, and $w_{22}$ acts by twisted complex conjugation.
The fixed field of  $H_R$ under $\langle w_{13}, w_{22} \rangle $ is quadratic, and can be computed to be $\Q(\sqrt{13})$. 
We have $\#\CM(R) = 16$ and thus $\#\overline{\CM}(R) =2$.
\end{example}

\subsection{Matching points and elliptic points}
Fix a basis $f_1,\dots, f_g$ for $S_2(\Gamma_0^*(N))$.
The canonical map (with respect to this basis)  for $X_0^*(N)$ is the map
\[ \tau \mapsto [f_1(\tau): \dots : f_g(\tau)]\]
where $q = e^{2 \pi i \tau}$.
One can read more about constructing these canonical models in \cite{GalbraithPhD}.

Given a quadratic order $R =\Z[\gamma]$, a point $\tau \in \calH$ satisfies $\tau \in \CM(R)$ if and only if there is an optimal embedding $\varphi : R \hookrightarrow \calO$, such that $\varphi(\gamma) \tau = \tau$. We may therefore enumerate the conjugacy classes of these optimal embeddings to enumerate $
\CM(R)$. Given $\tau \in \CM(R)$, whose stabilizer in $\Gamma_0^*(N)/\{ \pm 1\}$ is of order $h$, we compute its image as 
$$
\tau \mapsto \left[ \vartheta^{(h-1)}f_1(q): \ldots : \vartheta^{(h-1)}f_g(q)  \right],
$$
where $\vartheta = q \frac{d}{dq}$. Indeed, the uniformizer at $\tau$ is of the form $t = w^h$, and $\omega_f = \phi(t) \ dt$ with $dt = hw^{h-1} \ dw$. Thus, $f$ vanishes to degree $h-1$ at $\tau$ and differentiation $h-1$ times retrieves the values up to a scaling factor. 

In our code, matching is  done numerically to a high degree of precision, but not rigorously. To certify the matches, one could use Schofer's formula to obtain norms of values of Borcherds forms at CM points as in \cite{Err11, GY17}, or use the $j$-map from $X_0(N) \to X(1)$. 
Both approaches have significant computational costs as the level increases.

\section{Explaining exceptional points}
\label{sec:explanation}

In this section, we study the known exceptional points on star curves $X = X_0^*(N)$ of squarefree level and genus $g \geq 2$. In \cite{Genus2ANTS}, it is proven that there are exactly 51 exceptional points in genus 2. In the previous section, we found ten exceptional points in genus 3, four exceptional points in genus 4, and zero exceptional points in higher genus. Altogether, there are 65 known exceptional points. 

We describe three simple geometric constructions ( ``automorphisms'', ``collinearity'', and ``elliptic covers'') that will ultimately  ``explain'' all 65 points. All three constructions start with a natural algebraic family of effective divisors $\{D_y\}_{y \in Y}$ on $X$, indexed by a variety $Y$ over $\Q$. If for some $y \in Y(\Q)$ we have $D_y = x + D'$, where $x \in X$ is degree $1$ and $D'$ is supported on special closed points (cusps or CM points), then $x$ is a rational point and we declare it ``explained''. 
The choice of $Y$ in the three constructions is 
\[
Y = 
\begin{cases}
    X \simeq \Gamma_\alpha & \mbox{ automorphism } \alpha \colon X \to X\\
    \P^{g-1} = \P H^0(X, \omega_X) & \mbox{ collinearity}\\
    E & \mbox{ elliptic cover } X \to E.
\end{cases}
\]
We spell this out in more detail below.

\subsection{Automorphisms}\label{subsec: automorphisms}
\begin{definition}
    A point $x \in X_0^*(N)(\Q)$ is {\em explained by automorphism} if there exists $\alpha \in \Aut(X_0^*(N))$ and a special point $y \in X_0^*(N)(\Q)$ such that $\alpha(y) = x$. 
\end{definition} 

Here, the family of divisors is the degree 2 family $\{x+\alpha(x)\}_{x \in X}$ indexed by $X$ itself.

This notion of ``explanation'' goes back to Ogg, who observed that the non-cuspidal points on the genus two curve $X_0(37)$ are explained by the hyperelliptic involution. 

More recently, \cite{Genus2ANTS} observed that $49$ of the $51$ exceptional points on genus two curves $X_0^*(N)$ are explained by automorphism, the exceptions being a pair of exceptional points on $X_0^*(286)$. Of these, $47$ are explained by the hyperelliptic involution, while two points on $X_0^*(129)$ are explained by bielliptic involutions.

In genus $3$, we find that 8 out of the 10 exceptional points are explained by a bielliptic involutions, and the remaining 2 have no automorphisms (see Table  \ref{tab:exceptional_points_genus_3} in the ``Aut.'' column).  
In genus $4$, only 1 out of the 4 exceptional points are explained by an automorphism, which is also a bielliptic involution (see Table \ref{tab:exceptional_points_genus_4}  in the ``Aut.'' column). 
Bars and Gonz\'alez \cite{BarsGonzalezAut} computed automorphism groups of $X_0^*(N)$ for squarefree $N$ and found all cases when they are nontrivial.
In each case, we also  checked that these nontrivial automorphisms did not produce more exceptional points than the ones we already found.
Thus, automorphisms explain 9 out of the 14 known exceptional points in genus $g \geq 3$.

\begin{remark}
    One might wonder why there are so many bielliptic genus $3$ and $4$ curves $X_0^*(N)$ to begin with. For example, 11 out of the $31$ genus 3 curves $X_0^*(N)$ are bielliptic \cite{BarsGonzalezRovira}. 
   Proposition  \ref{prop: elliptic cover info and degree} below explains this phenomenon. 
\end{remark}

\subsection{Collinearity}\label{subsec: collinearity}

Galbraith observed in \cite{Galbraithpplus} that the genus 4 curves $X_0^*(137)$ and $X_0^*(311)$ each have an exceptional point. Their automorphism groups are trivial, so the previous construction does not apply.

Instead, we use ``collinearity'', as introduced by Derickx--Hashimoto--Najman--Shnidman in their ``explanation'' of all rational points on all modular curves \cite{DHNS}.  Let $K_X$ denote a canonical divisor on $X = X_0^*(N)$, and let $\equiv$ denote linear equivalence.

\begin{definition}
    A point $x \in X_0^*(N)(\Q)$ is {\em explained by collinearity} if there exists special closed points $x_1,\ldots, x_k \in X_0^*(N)$ such that $x + \sum_{i = 1}^k x_i \equiv K_X$.
\end{definition}

Nothing in the definition requires the special divisor to be unique, and in
practice it is not: most of the exceptional points below lie on several
special hyperplanes at once (Table~\ref{tab:all_planes}).

Here, the family of divisors is the canonical linear system $|K_X| = \P H^0(X, \omega_X) \simeq \P^{g-1}$ consisting of divisors linearly equivalent to $K_X$. This has degree $2g-2$.

\begin{example}\label{rem: nonhyp vs hyp}
    If $g = 2$, then $\deg(K_X) = 2$ and we see immediately that a point is explained by collinearity if and only if  it is explained by the hyperelliptic involution.  
\end{example}

\begin{example}\label{ex: g large}
    Suppose $X$ is non-hyperelliptic, so that it embeds in $\P^{g-1} = |K_X|$.  If $g = 3$, the definition reduces to the following: if $P,Q,R$ are special points on the quartic plane curve $X \subset \P_\Q^2$ that all happen to lie on a line $L \subset \P_\Q^2$, then the fourth point on $L \cap X$ is rational and is explained. Hence the name ``collinearity''. Two points in $\P^2$ determine a unique line, so collinearity amounts to the single ``coincidence'' that $R$ lies on the line $PQ$.  In general,  collinearity requires a  collection of $2g-3$ special points to lie in a common $\Q$-hyperplane.
    Since $g-1$ points in general position determine a unique hyperplane, this amounts to $g-2$ ``coincidences''. The final $(2g-2)$-th point may itself be special, but this is not forced. 
\end{example}

\begin{remark}\label{rem: collinearity as automorphism}
    Example \ref{rem: nonhyp vs hyp} seems special to genus $2$,  but collinearity is closely related to automorphisms in  {\it all} genera.  By Riemann--Roch,  the formula $D \mapsto K_X - D$ defines a birational involution of the image of $\Sym^{g-1}(X)$ in $\Pic^{g-1}_X$. Thus, in the context of exceptional points on $\Sym^{g-1}(X)$, ``collinearity'' is a special case of ``automorphism''.  (By Faltings' theorem, there is a proper closed subvariety $Z \subset \Sym^{g-1}(X)$ such that there are  only finitely many $K$-points on $\Sym^{g-1}(X) \setminus Z$, for every number field $K$, and it makes sense to call the non-special points in $(\Sym^{g-1}(X) \setminus Z)(\Q)$ exceptional.) 
\end{remark}

\begin{remark}
    By the theory of complex multiplication, there are only finitely many special points $x_i \in X_0^*(N)$ of degree at most $2g-3$, so the points explained by collinearity are in principle computable. 
\end{remark}

    Using the algorithms developed in the previous section, we can verify (to a high degree of numerical precision) that all known genus 3 and 4 exceptional points are explained by collinearity. See Tables \ref{tab:exceptional_points_genus_3} and \ref{tab:exceptional_points_genus_4} for the CM points appearing in the collinearity divisors.

In particular, both of Galbraith's exceptional points are explained by collinearity. Because of their historical interest, we certified this rigorously using the $j$-map $X_0(N) \to X(1)$.
\begin{proposition}
The exceptional point on $X_0^*(137)$ is explained by collinearity with the CM points of discriminants $-32, -11, -4$, and the cusp.

The exceptional point on $X_0^*(311)$ is explained by collinearity with the CM points of discriminants $-232$, $-123$, and $-19$. 
\end{proposition}
\begin{remark}
The collinearity relations in Tables \ref{tab:exceptional_points_genus_3} and \ref{tab:exceptional_points_genus_4} are generally not unique. For example at level $N = 137$, the exceptional point is also coplanar with the CM points corresponding to
$\{-112, -19, -11, -7\} $
$\{ -427, -19, -4 \}$, 
$\{ -72, -7, -4 \}$, and
$\{ -112, -8, \text{cusp} \}$.
The latter three planes have tangencies at the  points $-19$, $-7$ and the cusp, respectively.
We rigorously certified all of these collinearity relations on $X_0^*(137)$.
Table~\ref{tab:all_planes} lists every special hyperplane we can confirm
through each exceptional point in genus $3$ and $4$: there are $34$ of them
across the $14$ points, and $10$ of the $14$ levels carry more than one.
We emphasize that this multiplicity is a statement about the curves and not about the search.  
\end{remark}

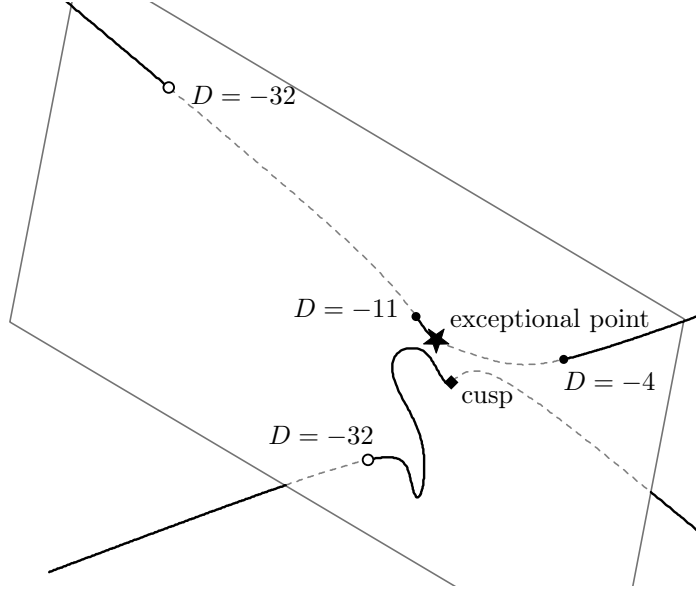
\begin{figure}
\begin{tikzpicture}[line cap=round, line join=round,
    curve/.style={black, line width=0.9pt},
    behind/.style={black!50, line width=0.6pt, dash pattern=on 2.2pt off 2.4pt},
    ratpt/.style={circle, fill=black, minimum size=3.4pt, inner sep=0pt},
    quadpt/.style={circle, draw=black, fill=white, line width=0.7pt, minimum size=4pt, inner sep=0pt},
    cusppt/.style={diamond, fill=black, inner sep=1.4pt},
    excpt/.style={star, star points=5, star point ratio=0.4, draw=black, fill=black, minimum size=4pt, inner sep=0pt},
    lbl/.style={font=\footnotesize, align=center}]
\begin{scope}
\clip (-5.32,-3.45) rectangle (3.90,4.25);
\draw[behind] plot coordinates {(-1.52,-2.13) (-1.50,-2.12) (-1.49,-2.12) (-1.47,-2.11) (-1.45,-2.10) (-1.44,-2.10) (-1.42,-2.09) (-1.40,-2.09) (-1.39,-2.08) (-1.37,-2.08) (-1.35,-2.07) (-1.34,-2.06) (-1.32,-2.06) (-1.30,-2.05) (-1.28,-2.05) (-1.27,-2.04) (-1.25,-2.04) (-1.23,-2.03) (-1.22,-2.02) (-1.20,-2.02) (-1.18,-2.01) (-1.17,-2.01) (-1.15,-2.00) (-1.13,-2.00) (-1.12,-1.99) (-1.10,-1.99) (-1.08,-1.98) (-1.07,-1.97) (-1.05,-1.97) (-1.03,-1.96) (-1.02,-1.96) (-1.00,-1.95) (-0.98,-1.95) (-0.96,-1.94) (-0.95,-1.94) (-0.93,-1.93) (-0.91,-1.93) (-0.90,-1.92) (-0.88,-1.92) (-0.86,-1.91) (-0.85,-1.91) (-0.83,-1.90) (-0.81,-1.90) (-0.80,-1.89) (-0.78,-1.89) (-0.76,-1.88) (-0.75,-1.88) (-0.73,-1.87) (-0.71,-1.87) (-0.70,-1.86) (-0.68,-1.86) (-0.66,-1.85) (-0.65,-1.85) (-0.63,-1.84) (-0.61,-1.84) (-0.60,-1.83) (-0.58,-1.83) (-0.56,-1.82) (-0.55,-1.82) (-0.53,-1.82) (-0.51,-1.81) (-0.50,-1.81) (-0.48,-1.80) (-0.46,-1.80) (-0.45,-1.80) (-0.43,-1.79)};
\draw[behind] plot coordinates {(0.69,-0.75) (0.71,-0.73) (0.73,-0.71) (0.76,-0.70) (0.78,-0.68) (0.80,-0.67) (0.83,-0.65) (0.85,-0.64) (0.87,-0.63) (0.90,-0.63) (0.92,-0.62) (0.95,-0.62) (0.98,-0.62) (1.00,-0.62) (1.03,-0.62) (1.06,-0.62) (1.08,-0.63) (1.11,-0.64) (1.14,-0.65) (1.16,-0.65) (1.19,-0.66) (1.21,-0.67) (1.24,-0.69) (1.27,-0.70) (1.29,-0.71) (1.32,-0.72) (1.34,-0.74) (1.37,-0.75) (1.39,-0.76) (1.42,-0.78) (1.44,-0.79) (1.47,-0.81) (1.49,-0.82) (1.51,-0.84) (1.54,-0.85) (1.56,-0.87) (1.59,-0.89) (1.61,-0.90) (1.63,-0.92) (1.66,-0.93) (1.68,-0.95) (1.70,-0.97) (1.73,-0.98) (1.75,-1.00) (1.78,-1.01) (1.80,-1.03) (1.82,-1.05) (1.85,-1.06) (1.87,-1.08) (1.89,-1.10) (1.91,-1.12) (1.94,-1.13) (1.96,-1.15) (1.98,-1.17) (2.01,-1.18) (2.03,-1.20) (2.05,-1.22) (2.08,-1.24) (2.10,-1.25) (2.12,-1.27) (2.14,-1.29) (2.17,-1.31) (2.19,-1.32) (2.21,-1.34) (2.24,-1.36) (2.26,-1.38) (2.28,-1.39) (2.30,-1.41) (2.33,-1.43) (2.35,-1.45) (2.37,-1.46) (2.39,-1.48) (2.42,-1.50) (2.44,-1.52) (2.46,-1.54) (2.48,-1.55) (2.51,-1.57) (2.53,-1.59) (2.55,-1.61) (2.57,-1.62) (2.60,-1.64) (2.62,-1.66) (2.64,-1.68) (2.66,-1.70) (2.69,-1.71) (2.71,-1.73) (2.73,-1.75) (2.75,-1.77) (2.78,-1.79) (2.80,-1.80) (2.82,-1.82) (2.84,-1.84) (2.87,-1.86) (2.89,-1.88) (2.91,-1.90) (2.93,-1.91) (2.96,-1.93) (2.98,-1.95) (3.00,-1.97) (3.02,-1.99) (3.05,-2.00) (3.07,-2.02) (3.09,-2.04) (3.11,-2.06) (3.13,-2.08) (3.16,-2.09) (3.18,-2.11) (3.20,-2.13) (3.22,-2.15) (3.25,-2.17) (3.27,-2.19) (3.29,-2.20) (3.31,-2.22)};
\draw[behind] plot coordinates {(-3.05,3.12) (-3.03,3.10) (-3.01,3.08) (-2.98,3.06) (-2.96,3.04) (-2.94,3.02) (-2.92,3.00) (-2.90,2.99) (-2.87,2.97) (-2.85,2.95) (-2.83,2.93) (-2.81,2.91) (-2.79,2.89) (-2.76,2.87) (-2.74,2.86) (-2.72,2.84) (-2.70,2.82) (-2.67,2.80) (-2.65,2.78) (-2.63,2.76) (-2.61,2.74) (-2.59,2.73) (-2.56,2.71) (-2.54,2.69) (-2.52,2.67) (-2.50,2.65) (-2.48,2.63) (-2.45,2.61) (-2.43,2.60) (-2.41,2.58) (-2.39,2.56) (-2.37,2.54) (-2.34,2.52) (-2.32,2.50) (-2.30,2.48) (-2.28,2.47) (-2.26,2.45) (-2.23,2.43) (-2.21,2.41) (-2.19,2.39) (-2.17,2.37) (-2.15,2.35) (-2.12,2.34) (-2.10,2.32) (-2.08,2.30) (-2.06,2.28) (-2.04,2.26) (-2.02,2.24) (-1.99,2.22) (-1.97,2.20) (-1.95,2.19) (-1.93,2.17) (-1.91,2.15) (-1.88,2.13) (-1.86,2.11) (-1.84,2.09) (-1.82,2.07) (-1.80,2.06) (-1.77,2.04) (-1.75,2.02) (-1.73,2.00) (-1.71,1.98) (-1.69,1.96) (-1.66,1.94) (-1.64,1.92) (-1.62,1.91) (-1.60,1.89) (-1.58,1.87) (-1.56,1.85) (-1.53,1.83) (-1.51,1.81) (-1.49,1.79) (-1.47,1.77) (-1.45,1.76) (-1.42,1.74) (-1.40,1.72) (-1.38,1.70) (-1.36,1.68) (-1.34,1.66) (-1.32,1.64) (-1.29,1.62) (-1.27,1.60) (-1.25,1.59) (-1.23,1.57) (-1.21,1.55) (-1.19,1.53) (-1.16,1.51) (-1.14,1.49) (-1.12,1.47) (-1.10,1.45) (-1.08,1.43) (-1.06,1.41) (-1.03,1.40) (-1.01,1.38) (-0.99,1.36) (-0.97,1.34) (-0.95,1.32) (-0.93,1.30) (-0.90,1.28) (-0.88,1.26) (-0.86,1.24) (-0.84,1.22) (-0.82,1.20) (-0.80,1.18) (-0.78,1.17) (-0.76,1.15) (-0.73,1.13) (-0.71,1.11) (-0.69,1.09) (-0.67,1.07) (-0.65,1.05) (-0.63,1.03) (-0.61,1.01) (-0.59,0.99) (-0.57,0.97) (-0.54,0.95) (-0.52,0.93) (-0.50,0.91) (-0.48,0.89) (-0.46,0.87) (-0.44,0.85) (-0.42,0.83) (-0.40,0.81) (-0.38,0.79) (-0.36,0.77) (-0.34,0.75) (-0.32,0.73) (-0.30,0.71) (-0.28,0.69) (-0.26,0.67) (-0.24,0.64) (-0.22,0.62) (-0.20,0.60) (-0.18,0.58) (-0.16,0.56) (-0.14,0.54) (-0.12,0.52) (-0.10,0.50) (-0.08,0.47) (-0.07,0.45) (-0.05,0.43) (-0.03,0.41) (-0.01,0.39) (0.01,0.36) (0.03,0.34) (0.04,0.32) (0.06,0.29) (0.08,0.27) (0.10,0.25) (0.11,0.23) (0.13,0.20) (0.15,0.18) (0.17,0.16) (0.18,0.13) (0.20,0.11) (0.22,0.08)};
\draw[behind] plot coordinates {(0.48,-0.21) (0.51,-0.23) (0.54,-0.24) (0.56,-0.25) (0.59,-0.26) (0.62,-0.28) (0.64,-0.29) (0.67,-0.30) (0.70,-0.31) (0.72,-0.32) (0.75,-0.33) (0.78,-0.34) (0.80,-0.35) (0.83,-0.36) (0.86,-0.37) (0.88,-0.37) (0.91,-0.38) (0.94,-0.39) (0.96,-0.40) (0.99,-0.41) (1.01,-0.42) (1.04,-0.43) (1.06,-0.43) (1.09,-0.44) (1.11,-0.45) (1.14,-0.45) (1.16,-0.46) (1.18,-0.47) (1.21,-0.47) (1.23,-0.48) (1.25,-0.49) (1.27,-0.49) (1.30,-0.50) (1.32,-0.50) (1.34,-0.50) (1.36,-0.51) (1.38,-0.51) (1.40,-0.52) (1.42,-0.52) (1.44,-0.52) (1.46,-0.52) (1.48,-0.53) (1.50,-0.53) (1.52,-0.53) (1.54,-0.53) (1.56,-0.53) (1.58,-0.53) (1.60,-0.53) (1.62,-0.54) (1.64,-0.54) (1.66,-0.54) (1.67,-0.53) (1.69,-0.53) (1.71,-0.53) (1.73,-0.53) (1.75,-0.53) (1.77,-0.53) (1.79,-0.53) (1.80,-0.53) (1.82,-0.52) (1.84,-0.52) (1.86,-0.52) (1.88,-0.52) (1.89,-0.52) (1.91,-0.51) (1.93,-0.51) (1.95,-0.51) (1.96,-0.50) (1.98,-0.50) (2.00,-0.50) (2.02,-0.50) (2.03,-0.49) (2.05,-0.49) (2.07,-0.48) (2.09,-0.48) (2.10,-0.48) (2.12,-0.47) (2.14,-0.47) (2.16,-0.47)};
\draw[black!55, line width=0.6pt] (-4.24,4.82) -- (3.75,0.07) -- (2.83,-4.72) -- (-5.17,0.03) -- cycle;
\draw[curve] plot coordinates {(-4.65,-3.28) (-4.64,-3.27) (-4.62,-3.27) (-4.60,-3.26) (-4.59,-3.25) (-4.57,-3.25) (-4.55,-3.24) (-4.54,-3.23) (-4.52,-3.23) (-4.50,-3.22) (-4.48,-3.22) (-4.47,-3.21) (-4.45,-3.20) (-4.43,-3.20) (-4.42,-3.19) (-4.40,-3.18) (-4.38,-3.18) (-4.37,-3.17) (-4.35,-3.16) (-4.33,-3.16) (-4.32,-3.15) (-4.30,-3.15) (-4.28,-3.14) (-4.26,-3.13) (-4.25,-3.13) (-4.23,-3.12) (-4.21,-3.11) (-4.20,-3.11) (-4.18,-3.10) (-4.16,-3.09) (-4.15,-3.09) (-4.13,-3.08) (-4.11,-3.08) (-4.09,-3.07) (-4.08,-3.06) (-4.06,-3.06) (-4.04,-3.05) (-4.03,-3.04) (-4.01,-3.04) (-3.99,-3.03) (-3.98,-3.03) (-3.96,-3.02) (-3.94,-3.01) (-3.93,-3.01) (-3.91,-3.00) (-3.89,-2.99) (-3.87,-2.99) (-3.86,-2.98) (-3.84,-2.97) (-3.82,-2.97) (-3.81,-2.96) (-3.79,-2.96) (-3.77,-2.95) (-3.76,-2.94) (-3.74,-2.94) (-3.72,-2.93) (-3.71,-2.92) (-3.69,-2.92) (-3.67,-2.91) (-3.65,-2.90) (-3.64,-2.90) (-3.62,-2.89) (-3.60,-2.89) (-3.59,-2.88) (-3.57,-2.87) (-3.55,-2.87) (-3.54,-2.86) (-3.52,-2.85) (-3.50,-2.85) (-3.48,-2.84) (-3.47,-2.84) (-3.45,-2.83) (-3.43,-2.82) (-3.42,-2.82) (-3.40,-2.81) (-3.38,-2.80) (-3.37,-2.80) (-3.35,-2.79) (-3.33,-2.78) (-3.32,-2.78) (-3.30,-2.77) (-3.28,-2.77) (-3.26,-2.76) (-3.25,-2.75) (-3.23,-2.75) (-3.21,-2.74) (-3.20,-2.73) (-3.18,-2.73) (-3.16,-2.72) (-3.15,-2.72) (-3.13,-2.71) (-3.11,-2.70) (-3.09,-2.70) (-3.08,-2.69) (-3.06,-2.68) (-3.04,-2.68) (-3.03,-2.67) (-3.01,-2.67) (-2.99,-2.66) (-2.98,-2.65) (-2.96,-2.65) (-2.94,-2.64) (-2.93,-2.63) (-2.91,-2.63) (-2.89,-2.62) (-2.87,-2.62) (-2.86,-2.61) (-2.84,-2.60) (-2.82,-2.60) (-2.81,-2.59) (-2.79,-2.58) (-2.77,-2.58) (-2.76,-2.57) (-2.74,-2.57) (-2.72,-2.56) (-2.71,-2.55) (-2.69,-2.55) (-2.67,-2.54) (-2.65,-2.53) (-2.64,-2.53) (-2.62,-2.52) (-2.60,-2.52) (-2.59,-2.51) (-2.57,-2.50) (-2.55,-2.50) (-2.54,-2.49) (-2.52,-2.48) (-2.50,-2.48) (-2.49,-2.47) (-2.47,-2.47) (-2.45,-2.46) (-2.43,-2.45) (-2.42,-2.45) (-2.40,-2.44) (-2.38,-2.44) (-2.37,-2.43) (-2.35,-2.42) (-2.33,-2.42) (-2.32,-2.41) (-2.30,-2.40) (-2.28,-2.40) (-2.27,-2.39) (-2.25,-2.39) (-2.23,-2.38) (-2.21,-2.37) (-2.20,-2.37) (-2.18,-2.36) (-2.16,-2.36) (-2.15,-2.35) (-2.13,-2.34) (-2.11,-2.34) (-2.10,-2.33) (-2.08,-2.33) (-2.06,-2.32) (-2.05,-2.31) (-2.03,-2.31) (-2.01,-2.30) (-1.99,-2.30) (-1.98,-2.29) (-1.96,-2.28) (-1.94,-2.28) (-1.93,-2.27) (-1.91,-2.26) (-1.89,-2.26) (-1.88,-2.25) (-1.86,-2.25) (-1.84,-2.24) (-1.83,-2.23) (-1.81,-2.23) (-1.79,-2.22) (-1.77,-2.22) (-1.76,-2.21) (-1.74,-2.20) (-1.72,-2.20) (-1.71,-2.19) (-1.69,-2.19) (-1.67,-2.18) (-1.66,-2.18) (-1.64,-2.17) (-1.62,-2.16) (-1.61,-2.16) (-1.59,-2.15) (-1.57,-2.15) (-1.55,-2.14) (-1.54,-2.13) (-1.52,-2.13)};
\draw[curve] plot coordinates {(-0.43,-1.79) (-0.41,-1.79) (-0.40,-1.79) (-0.38,-1.78) (-0.36,-1.78) (-0.35,-1.78) (-0.33,-1.77) (-0.31,-1.77) (-0.30,-1.77) (-0.28,-1.77) (-0.26,-1.77) (-0.25,-1.76) (-0.23,-1.76) (-0.22,-1.76) (-0.20,-1.76) (-0.18,-1.76) (-0.17,-1.76) (-0.15,-1.76) (-0.13,-1.76) (-0.12,-1.76) (-0.10,-1.76) (-0.09,-1.77) (-0.07,-1.77) (-0.06,-1.77) (-0.04,-1.78) (-0.03,-1.78) (-0.01,-1.79) (0.00,-1.80) (0.02,-1.81) (0.03,-1.82) (0.04,-1.83) (0.05,-1.85) (0.07,-1.87) (0.08,-1.88) (0.09,-1.90) (0.09,-1.92) (0.10,-1.94) (0.11,-1.97) (0.12,-1.99) (0.12,-2.01) (0.13,-2.04) (0.13,-2.06) (0.14,-2.09) (0.15,-2.11) (0.15,-2.13) (0.16,-2.16) (0.16,-2.18) (0.17,-2.20) (0.17,-2.22) (0.18,-2.24) (0.19,-2.25) (0.20,-2.27) (0.20,-2.28) (0.21,-2.29) (0.22,-2.29) (0.23,-2.29) (0.23,-2.29) (0.24,-2.28) (0.25,-2.27) (0.26,-2.26) (0.26,-2.25) (0.27,-2.24) (0.27,-2.22) (0.28,-2.20) (0.28,-2.19) (0.29,-2.17) (0.29,-2.15) (0.30,-2.13) (0.30,-2.11) (0.30,-2.09) (0.30,-2.06) (0.31,-2.04) (0.31,-2.02) (0.31,-2.00) (0.31,-1.97) (0.31,-1.95) (0.31,-1.93) (0.31,-1.90) (0.31,-1.88) (0.31,-1.85) (0.31,-1.83) (0.31,-1.81) (0.31,-1.78) (0.31,-1.76) (0.31,-1.73) (0.30,-1.71) (0.30,-1.68) (0.30,-1.65) (0.29,-1.63) (0.29,-1.60) (0.28,-1.58) (0.28,-1.55) (0.27,-1.53) (0.26,-1.50) (0.26,-1.48) (0.25,-1.45) (0.24,-1.43) (0.23,-1.40) (0.22,-1.38) (0.21,-1.35) (0.20,-1.33) (0.19,-1.30) (0.18,-1.28) (0.17,-1.25) (0.15,-1.23) (0.14,-1.20) (0.13,-1.18) (0.12,-1.15) (0.10,-1.13) (0.09,-1.10) (0.08,-1.08) (0.06,-1.05) (0.05,-1.03) (0.03,-1.00) (0.02,-0.98) (0.01,-0.95) (-0.00,-0.93) (-0.02,-0.90) (-0.03,-0.88) (-0.04,-0.85) (-0.05,-0.82) (-0.06,-0.80) (-0.07,-0.77) (-0.08,-0.74) (-0.09,-0.72) (-0.10,-0.69) (-0.10,-0.66) (-0.11,-0.63) (-0.11,-0.60) (-0.11,-0.57) (-0.11,-0.54) (-0.11,-0.52) (-0.10,-0.49) (-0.09,-0.46) (-0.08,-0.43) (-0.06,-0.41) (-0.05,-0.39) (-0.02,-0.37) (-0.00,-0.35) (0.02,-0.34) (0.05,-0.33) (0.08,-0.32) (0.11,-0.32) (0.13,-0.32) (0.16,-0.32) (0.19,-0.32) (0.22,-0.32) (0.25,-0.33) (0.27,-0.34) (0.30,-0.35) (0.32,-0.37) (0.34,-0.39) (0.36,-0.41) (0.38,-0.43) (0.40,-0.45) (0.41,-0.48) (0.43,-0.50) (0.44,-0.53) (0.46,-0.55) (0.47,-0.57) (0.48,-0.60) (0.50,-0.62) (0.51,-0.65) (0.52,-0.67) (0.54,-0.69) (0.55,-0.71) (0.56,-0.73) (0.58,-0.75) (0.59,-0.77) (0.61,-0.78) (0.63,-0.78) (0.65,-0.78) (0.67,-0.77) (0.69,-0.75)};
\draw[curve] plot coordinates {(3.31,-2.22) (3.33,-2.24) (3.36,-2.26) (3.38,-2.28) (3.40,-2.30) (3.42,-2.31) (3.45,-2.33) (3.47,-2.35) (3.49,-2.37) (3.51,-2.39) (3.54,-2.41) (3.56,-2.42) (3.58,-2.44) (3.60,-2.46) (3.62,-2.48) (3.65,-2.50) (3.67,-2.52) (3.69,-2.53) (3.71,-2.55) (3.74,-2.57) (3.76,-2.59) (3.78,-2.61) (3.80,-2.63) (3.82,-2.64) (3.85,-2.66) (3.87,-2.68) (3.89,-2.70) (3.91,-2.72) (3.94,-2.74) (3.96,-2.75) (3.98,-2.77) (4.00,-2.79) (4.02,-2.81) (4.05,-2.83) (4.07,-2.85) (4.09,-2.86) (4.11,-2.88) (4.14,-2.90) (4.16,-2.92) (4.18,-2.94) (4.20,-2.96) (4.22,-2.97) (4.25,-2.99) (4.27,-3.01) (4.29,-3.03) (4.31,-3.05) (4.33,-3.07) (4.36,-3.08) (4.38,-3.10) (4.40,-3.12) (4.42,-3.14) (4.45,-3.16) (4.47,-3.18) (4.49,-3.20) (4.51,-3.21) (4.53,-3.23) (4.56,-3.25) (4.58,-3.27) (4.60,-3.29) (4.62,-3.31) (4.65,-3.32) (4.67,-3.34) (4.69,-3.36) (4.71,-3.38) (4.73,-3.40) (4.76,-3.42) (4.78,-3.43) (4.80,-3.45) (4.82,-3.47) (4.84,-3.49) (4.87,-3.51) (4.89,-3.53) (4.91,-3.55) (4.93,-3.56) (4.96,-3.58) (4.98,-3.60) (5.00,-3.62) (5.02,-3.64) (5.04,-3.66) (5.07,-3.67) (5.09,-3.69) (5.11,-3.71) (5.13,-3.73) (5.15,-3.75) (5.18,-3.77) (5.20,-3.78) (5.22,-3.80) (5.24,-3.82) (5.27,-3.84) (5.29,-3.86) (5.31,-3.88) (5.33,-3.90) (5.35,-3.91) (5.38,-3.93) (5.40,-3.95) (5.42,-3.97) (5.44,-3.99) (5.46,-4.01) (5.49,-4.02) (5.51,-4.04) (5.53,-4.06) (5.55,-4.08) (5.58,-4.10) (5.60,-4.12) (5.62,-4.14) (5.64,-4.15) (5.66,-4.17) (5.69,-4.19) (5.71,-4.21) (5.73,-4.23) (5.75,-4.25) (5.77,-4.26) (5.80,-4.28) (5.82,-4.30) (5.84,-4.32) (5.86,-4.34) (5.89,-4.36) (5.91,-4.38) (5.93,-4.39) (5.95,-4.41) (5.97,-4.43) (6.00,-4.45) (6.02,-4.47) (6.04,-4.49) (6.06,-4.50) (6.08,-4.52) (6.11,-4.54) (6.13,-4.56) (6.15,-4.58) (6.17,-4.60) (6.20,-4.62) (6.22,-4.63) (6.24,-4.65) (6.26,-4.67) (6.28,-4.69) (6.31,-4.71) (6.33,-4.73) (6.35,-4.74) (6.37,-4.76) (6.39,-4.78) (6.42,-4.80) (6.44,-4.82) (6.46,-4.84) (6.48,-4.86) (6.51,-4.87) (6.53,-4.89) (6.55,-4.91) (6.57,-4.93) (6.59,-4.95) (6.62,-4.97) (6.64,-4.98) (6.66,-5.00) (6.68,-5.02) (6.70,-5.04) (6.73,-5.06) (6.75,-5.08) (6.77,-5.10) (6.79,-5.11) (6.82,-5.13) (6.84,-5.15) (6.86,-5.17) (6.88,-5.19) (6.90,-5.21) (6.93,-5.23) (6.95,-5.24) (6.97,-5.26) (6.99,-5.28) (7.01,-5.30) (7.04,-5.32) (7.06,-5.34) (7.08,-5.35) (7.10,-5.37) (7.13,-5.39) (7.15,-5.41) (7.17,-5.43) (7.19,-5.45) (7.21,-5.47) (7.24,-5.48) (7.26,-5.50) (7.28,-5.52) (7.30,-5.54) (7.32,-5.56) (7.35,-5.58) (7.37,-5.59) (7.39,-5.61) (7.41,-5.63) (7.44,-5.65)};
\draw[curve] plot coordinates {(-6.89,6.34) (-6.87,6.32) (-6.85,6.30) (-6.83,6.28) (-6.81,6.27) (-6.78,6.25) (-6.76,6.23) (-6.74,6.21) (-6.72,6.19) (-6.70,6.17) (-6.67,6.15) (-6.65,6.14) (-6.63,6.12) (-6.61,6.10) (-6.58,6.08) (-6.56,6.06) (-6.54,6.04) (-6.52,6.02) (-6.50,6.01) (-6.47,5.99) (-6.45,5.97) (-6.43,5.95) (-6.41,5.93) (-6.39,5.91) (-6.36,5.89) (-6.34,5.88) (-6.32,5.86) (-6.30,5.84) (-6.28,5.82) (-6.25,5.80) (-6.23,5.78) (-6.21,5.77) (-6.19,5.75) (-6.16,5.73) (-6.14,5.71) (-6.12,5.69) (-6.10,5.67) (-6.08,5.65) (-6.05,5.64) (-6.03,5.62) (-6.01,5.60) (-5.99,5.58) (-5.97,5.56) (-5.94,5.54) (-5.92,5.52) (-5.90,5.51) (-5.88,5.49) (-5.85,5.47) (-5.83,5.45) (-5.81,5.43) (-5.79,5.41) (-5.77,5.40) (-5.74,5.38) (-5.72,5.36) (-5.70,5.34) (-5.68,5.32) (-5.66,5.30) (-5.63,5.28) (-5.61,5.27) (-5.59,5.25) (-5.57,5.23) (-5.55,5.21) (-5.52,5.19) (-5.50,5.17) (-5.48,5.15) (-5.46,5.14) (-5.43,5.12) (-5.41,5.10) (-5.39,5.08) (-5.37,5.06) (-5.35,5.04) (-5.32,5.02) (-5.30,5.01) (-5.28,4.99) (-5.26,4.97) (-5.24,4.95) (-5.21,4.93) (-5.19,4.91) (-5.17,4.90) (-5.15,4.88) (-5.13,4.86) (-5.10,4.84) (-5.08,4.82) (-5.06,4.80) (-5.04,4.78) (-5.01,4.77) (-4.99,4.75) (-4.97,4.73) (-4.95,4.71) (-4.93,4.69) (-4.90,4.67) (-4.88,4.65) (-4.86,4.64) (-4.84,4.62) (-4.82,4.60) (-4.79,4.58) (-4.77,4.56) (-4.75,4.54) (-4.73,4.52) (-4.71,4.51) (-4.68,4.49) (-4.66,4.47) (-4.64,4.45) (-4.62,4.43) (-4.60,4.41) (-4.57,4.40) (-4.55,4.38) (-4.53,4.36) (-4.51,4.34) (-4.48,4.32) (-4.46,4.30) (-4.44,4.28) (-4.42,4.27) (-4.40,4.25) (-4.37,4.23) (-4.35,4.21) (-4.33,4.19) (-4.31,4.17) (-4.29,4.15) (-4.26,4.14) (-4.24,4.12) (-4.22,4.10) (-4.20,4.08) (-4.18,4.06) (-4.15,4.04) (-4.13,4.02) (-4.11,4.01) (-4.09,3.99) (-4.06,3.97) (-4.04,3.95) (-4.02,3.93) (-4.00,3.91) (-3.98,3.90) (-3.95,3.88) (-3.93,3.86) (-3.91,3.84) (-3.89,3.82) (-3.87,3.80) (-3.84,3.78) (-3.82,3.77) (-3.80,3.75) (-3.78,3.73) (-3.76,3.71) (-3.73,3.69) (-3.71,3.67) (-3.69,3.65) (-3.67,3.64) (-3.65,3.62) (-3.62,3.60) (-3.60,3.58) (-3.58,3.56) (-3.56,3.54) (-3.53,3.52) (-3.51,3.51) (-3.49,3.49) (-3.47,3.47) (-3.45,3.45) (-3.42,3.43) (-3.40,3.41) (-3.38,3.39) (-3.36,3.38) (-3.34,3.36) (-3.31,3.34) (-3.29,3.32) (-3.27,3.30) (-3.25,3.28) (-3.23,3.26) (-3.20,3.25) (-3.18,3.23) (-3.16,3.21) (-3.14,3.19) (-3.12,3.17) (-3.09,3.15) (-3.07,3.13) (-3.05,3.12)};
\draw[curve] plot coordinates {(0.22,0.08) (0.23,0.06) (0.25,0.04) (0.27,0.01) (0.28,-0.01) (0.30,-0.03) (0.32,-0.06) (0.33,-0.08) (0.35,-0.10) (0.37,-0.12) (0.39,-0.14) (0.41,-0.16) (0.44,-0.18) (0.46,-0.20) (0.48,-0.21)};
\draw[curve] plot coordinates {(2.16,-0.47) (2.17,-0.46) (2.19,-0.46) (2.21,-0.45) (2.23,-0.45) (2.24,-0.44) (2.26,-0.44) (2.28,-0.44) (2.30,-0.43) (2.31,-0.43) (2.33,-0.42) (2.35,-0.42) (2.36,-0.41) (2.38,-0.41) (2.40,-0.40) (2.42,-0.40) (2.43,-0.39) (2.45,-0.39) (2.47,-0.38) (2.48,-0.38) (2.50,-0.37) (2.52,-0.37) (2.54,-0.36) (2.55,-0.36) (2.57,-0.35) (2.59,-0.35) (2.60,-0.34) (2.62,-0.34) (2.64,-0.33) (2.66,-0.33) (2.67,-0.32) (2.69,-0.32) (2.71,-0.31) (2.72,-0.31) (2.74,-0.30) (2.76,-0.30) (2.78,-0.29) (2.79,-0.28) (2.81,-0.28) (2.83,-0.27) (2.84,-0.27) (2.86,-0.26) (2.88,-0.26) (2.89,-0.25) (2.91,-0.25) (2.93,-0.24) (2.95,-0.23) (2.96,-0.23) (2.98,-0.22) (3.00,-0.22) (3.01,-0.21) (3.03,-0.21) (3.05,-0.20) (3.07,-0.20) (3.08,-0.19) (3.10,-0.18) (3.12,-0.18) (3.13,-0.17) (3.15,-0.17) (3.17,-0.16) (3.18,-0.16) (3.20,-0.15) (3.22,-0.14) (3.24,-0.14) (3.25,-0.13) (3.27,-0.13) (3.29,-0.12) (3.30,-0.11) (3.32,-0.11) (3.34,-0.10) (3.35,-0.10) (3.37,-0.09) (3.39,-0.08) (3.41,-0.08) (3.42,-0.07) (3.44,-0.07) (3.46,-0.06) (3.47,-0.06) (3.49,-0.05) (3.51,-0.04) (3.52,-0.04) (3.54,-0.03) (3.56,-0.03) (3.58,-0.02) (3.59,-0.01) (3.61,-0.01) (3.63,-0.00) (3.64,0.00) (3.66,0.01) (3.68,0.02) (3.69,0.02) (3.71,0.03) (3.73,0.03) (3.75,0.04) (3.76,0.05) (3.78,0.05) (3.80,0.06) (3.81,0.06) (3.83,0.07) (3.85,0.08) (3.86,0.08) (3.88,0.09) (3.90,0.10) (3.91,0.10) (3.93,0.11) (3.95,0.11) (3.97,0.12) (3.98,0.13) (4.00,0.13) (4.02,0.14) (4.03,0.14) (4.05,0.15) (4.07,0.16) (4.08,0.16) (4.10,0.17) (4.12,0.17) (4.14,0.18) (4.15,0.19) (4.17,0.19) (4.19,0.20) (4.20,0.21) (4.22,0.21) (4.24,0.22) (4.25,0.22) (4.27,0.23) (4.29,0.24) (4.31,0.24) (4.32,0.25) (4.34,0.25) (4.36,0.26) (4.37,0.27) (4.39,0.27) (4.41,0.28) (4.42,0.29) (4.44,0.29) (4.46,0.30) (4.48,0.30) (4.49,0.31) (4.51,0.32) (4.53,0.32) (4.54,0.33) (4.56,0.34) (4.58,0.34) (4.59,0.35) (4.61,0.35) (4.63,0.36) (4.64,0.37) (4.66,0.37) (4.68,0.38) (4.70,0.39) (4.71,0.39) (4.73,0.40) (4.75,0.40) (4.76,0.41) (4.78,0.42) (4.80,0.42) (4.81,0.43) (4.83,0.44) (4.85,0.44) (4.87,0.45) (4.88,0.45) (4.90,0.46) (4.92,0.47) (4.93,0.47) (4.95,0.48) (4.97,0.49) (4.98,0.49) (5.00,0.50) (5.02,0.50) (5.04,0.51) (5.05,0.52) (5.07,0.52) (5.09,0.53) (5.10,0.54) (5.12,0.54) (5.14,0.55) (5.15,0.55) (5.17,0.56) (5.19,0.57) (5.20,0.57) (5.22,0.58) (5.24,0.59) (5.26,0.59) (5.27,0.60) (5.29,0.61) (5.31,0.61) (5.32,0.62) (5.34,0.62) (5.36,0.63) (5.37,0.64) (5.39,0.64) (5.41,0.65) (5.43,0.66) (5.44,0.66) (5.46,0.67) (5.48,0.67) (5.49,0.68) (5.51,0.69) (5.53,0.69) (5.54,0.70) (5.56,0.71) (5.58,0.71) (5.59,0.72) (5.61,0.73) (5.63,0.73) (5.65,0.74) (5.66,0.74) (5.68,0.75) (5.70,0.76) (5.71,0.76) (5.73,0.77) (5.75,0.78) (5.76,0.78) (5.78,0.79) (5.80,0.80) (5.82,0.80) (5.83,0.81) (5.85,0.81) (5.87,0.82) (5.88,0.83) (5.90,0.83) (5.92,0.84) (5.93,0.85) (5.95,0.85) (5.97,0.86) (5.99,0.86) (6.00,0.87) (6.02,0.88) (6.04,0.88) (6.05,0.89) (6.07,0.90) (6.09,0.90) (6.10,0.91) (6.12,0.92) (6.14,0.92) (6.15,0.93) (6.17,0.93) (6.19,0.94) (6.21,0.95) (6.22,0.95) (6.24,0.96)};
\end{scope}
\node[cusppt] (cusp) at (0.669,-0.769) {};
\node[ratpt] (D4) at (2.157,-0.465) {};
\node[ratpt] (D11) at (0.203,0.102) {};
\node[excpt] (exc) at (0.471,-0.206) {};
\node[quadpt] (D32a) at (-0.432,-1.794) {};
\node[quadpt] (D32b) at (-3.068,3.132) {};
\node[lbl, anchor=west] at (0.66,-1.00) {cusp};
\node[lbl, anchor=west] at (2.01,-0.76) {$D=-4$};
\node[lbl, anchor=east] at (0.12,0.20) {$D=-11$};
\node[lbl, anchor=west] at (0.51,0.05) {exceptional point};
\node[lbl, anchor=east] at (-0.21,-1.48) {$D=-32$};
\node[lbl, anchor=west] at (-2.92,3.02) {$D=-32$};
\end{tikzpicture}
\caption{Collinearity on $X_0^*(137)$. The curve has model
$ wx - yw + y^2 - yz = x^3 - 2(y-z)x^2 - (y^2+yz-2z^2)x + w^2z + (y+2z)wz + z^3 = 0$ 
in $\P^3$. Intersecting with the plane $x + 2z + 2w - y = 0$ spanned by the cusp and the rational CM points of discriminants $-11$ and $-4$, we find three other points, two of which happen to be conjugate quadratic points with CM of discriminant $-32$. The third point is therefore rational; it is Galbraith's exceptional point. The curve is drawn in the affine chart $z - w = 1.$  
}
\label{fig:137}
\end{figure}

  \begin{table}[ht]
  \centering
  \small
  \setlength{\tabcolsep}{5pt}
  \begin{tabular}{ccl@{\hspace{2.5em}}ccl}
  \toprule
  $N$ & $g$ & Collinearity & $N$ & $g$ & Collinearity \\
  \midrule
  $178$ & $3$ & $-136$, $-72$              & $329$ & $3$ & $-203$, cusp \\
        &     & $-180$, $-36$              &       &     & $-52$,
  $-35$$^{\dagger}$ \\
        &     & $-388$, $-40$$^{\dagger}$  & $430$ & $3$ & $-220$,
  $-120$$^{\dagger}$ \\
        &     & $-292$, $-100$             &       &     & $-820$, cusp \\
        &     & $-64$, cusp                &       &     & $-260$, $-20$ \\
  $183$ & $3$ & $-483$, $-75$$^{\dagger}$  & $455$ & $3$ & $-819$,
  $-195$$^{\dagger}$ \\
        &     & $-408$, $-48$              &       &     & $-595$, $-259$ \\
  $246$ & $3$ & $-564$, $-36$              & $510$ & $3$ & $-480$,
  cusp$^{\dagger}$ \\
        &     & $-264$, $-168$$^{\dagger}$ & $137$ & $4$ & $-32$, $-11$, $-4$,
  cusp$^{\dagger}$ \\
        &     & $-39$, $-20$               &       &     & $-112$, $-19$, $-11$,
  $-7$ \\
  $290$ & $3$ & $-1060$, $-180$            &       &     & $-427$, $-19$, $-4$
  \\
        &     & $-64$, $-24$$^{\dagger}$   &       &     & $-72$, $-7$, $-4$ \\
        &     & $-1240$, $-100$            &       &     & $-112$, $-8$, cusp \\
  $310$ & $3$ & $-55$, cusp$^{\dagger}$    & $311$ & $4$ & $-232$, $-123$,
  $-19$$^{\dagger}$ \\
        &     & $-136$, $-120$             & $370$ & $4$ & $-340$, $-260$,
  $-16$, cusp$^{\dagger}$ \\
        &     & $-840$, $-15$              & $399$ & $4$ & $-1995$, $-483$,
  $-147$, $-84$$^{\dagger}$ \\
  $318$ & $3$ & $-600$, $-312$             &       &     & \\
        &     & $-852$, $-372$$^{\dagger}$ &       &     & \\
  \bottomrule
  \end{tabular}
  \caption{Every confirmed special hyperplane through an exceptional
  point of $X_0^*(N)$ in genus $3$ and $4$, computed at
  $\texttt{eval\_prec} = 7000$.  Each collinearity is listed by the CM
  discriminants (and the cusp, where it occurs) of the special divisor
  cut out.  A dagger marks the plane recorded in
  Tables~\ref{tab:exceptional_points_genus_3}
  and~\ref{tab:exceptional_points_genus_4}.}
  \label{tab:all_planes}
  \end{table}

    Casta\~{n}o-Bernard \cite{CastanoBernard} had already observed some interesting collinearity properties of these points. 
    Part of the inspiration for this project was discovering that {\it all five} of  the collinear points (excluding the exceptional point) in the Galbraith examples are special points, and so the collinearity principle of \cite{DHNS} applies.  See Figure \ref{fig:137}.

As indicated in Example \ref{ex: g large}, collinearity becomes unlikely for $g \geq 5$, since  $2g-3 \geq 7$ special points must lie in the same $\Q$-hyperplane $H$ and the remaining $2g-2$-th point on $H$ must also be non-special. This is consistent with our empirical results for star curves of genus $g \geq 5$. 
That $4$ out of $36$ genus four curves  $X_0^*(N)$ have such collections of size $5$ is itself remarkable.

In total, $61$ out of the $65$ known exceptional points on Atkin--Lehner quotients of squarefree level are explained by collinearity.

\subsection{Elliptic covers}\label{subsec: elliptic covers}

 Together, automorphisms and collinearity explain 63 of the 65 known exceptional points on the curves $X_0^*(N)$. The only exceptions are a pair of points on the genus 2 curve $X_0^*(286)$ that form an orbit for $\Aut(X_0^*(286)) \simeq \Z/2\Z$. This mysterious pair has already been highlighted by \cite{Genus2ANTS}. 

Our explanation for these points will simultaneously give new explanations for most of the other exceptional points as well. It is based on the fact that star curves of low composite level tend to have low degree maps to elliptic curves. 

\begin{definition}\label{def: elliptic cover}
    A point $x \in X_0^*(N)(\Q)$ is {\it explained by elliptic cover} if there is an elliptic curve $E$ and a  map $\pi \colon X_0^*(N) \to E$ such that the fiber $\pi^{-1}(\pi(x))$, taken set-theoretically, contains only special points and $x$ (a special point may occur with multiplicity, as happens in the
    ramified cases of Table~\ref{tab:exceptional_cm_fibers}).
\end{definition}

In this construction, the family of divisors is $\{\pi^{-1}(e)\}_{e \in E}$, all of degree $\deg(\pi)$. 

Post-composing $\pi$ with an isogeny or a translation does not add explanatory power. Thus, we may assume that $\pi$ belongs to the set $\Hom_{\mathrm{min}}(X_0^*(N),E)$ of maps that do not factor through an isogeny of degree $> 1$ and that send the cusp to the zero element $O_E$.

\begin{remark}
    For a general pointed curve $X$, the set $\mathrm{Hom}_{\mathrm{min}}(X,E)$ may be infinite. 
    However, we show below (Proposition~\ref{prop: triple cover classification}) that $\Hom_{\mathrm{min}}(X_0^*(N),E)$ is finite for all star curves of squarefree level and all $E$. Thus, there are finitely many elliptic factors $E$ of $\mathrm{Jac}(X_0^*(N))$, and each $E$ can explain only finitely many exceptional points, since there are finitely many special points of bounded degree.
\end{remark}

To explain exceptional points using elliptic cover, we must understand the elliptic covers out of a given star curve and their degrees. Degree two amounts to explanation via bielliptic involution, which we have already discussed. In the next section, we classify all elliptic triple covers $X_0^*(N) \to E$ with $N$ squarefree. In particular, we will find that the two remaining exceptional points on $X_0^*(286)$ are indeed explained by elliptic triple cover. 

Thus, we conclude that {\it all known exceptional points on star curves $X_0^*(N)$ of squarefree level are explained}, either by automorphism, collinearity, or elliptic cover. We do not state this as a formal theorem, as we have only checked the collinearity explanations up to a high level of precision. Nonetheless, this is the second main result of this paper.

\section{Triple covers of elliptic curves}
Throughout this section, elliptic curves are identified by their LMFDB labels \cite{lmfdb}, hyperlinked to the corresponding pages. These differ in general from Cremona labels, which our code reports: the curve \ecl{370.b.3} is
\texttt{370a1} in Cremona's notation.

Before determining the possible elliptic covers emanating from $X_0^*(N)$ and their degrees, we begin with a lemma.

\begin{lemma}\label{lem: obstruction}
Let $d$ be a squarefree integer, let $f$ be a newform of level $d$ with rational
coefficients and all Atkin--Lehner eigenvalues $+1$, and let $\varphi_d \colon
X_0(d) \to E_f$ be the corresponding optimal quotient, normalized so that
$\varphi_d(\infty) = O_{E_f}$.  Then, for each $n \mid d$, there is a point 
\[ c_n \colonequals \varphi_{d,*}\bigl[ (w_n\infty) - (\infty) \bigr] \in E_f[2](\Q) \quad \text{ such that } \quad
  \varphi_d \circ w_n = \varphi_d + c_n.
\]
Moreover, the assignment $n \mapsto c_n$ is a homomorphism $c \colon W(d) \to E_f(\Q)[2]$.  In
particular $c = 0$ whenever $E_f(\Q)[2] = 0$.
\end{lemma}
 
\begin{proof}
By the normalization $\varphi_d(\infty) = O_{E_f}$ we have 
$\varphi_d(P)=\varphi_{d,*}\bigl[(P)-(\infty)\bigr]$.
The hypothesis $\varepsilon_n(f) = +1$ gives $\varphi_{d,*} \circ w_{n,*} =
\varphi_{d,*}$ on $J_0(d)$, so for any $P \in X_0(d)$, we see
\[
  \varphi_d(w_n P)
  = \varphi_{d,*}\bigl[ (w_nP) - (w_n\infty) \bigr] + c_n
  = \varphi_{d,*}\, w_{n,*}\bigl[ (P) - (\infty) \bigr] + c_n
  = \varphi_d(P) + c_n .
\]
Write $m_1 \star m_2 \colonequals m_1m_2/\gcd(m_1,m_2)^2$ for the group law on
the divisors of $N$, under which $m \mapsto w_m$ is an isomorphism onto $W(N)$.
Applying the above equality to  $P = w_{n_2}\infty$ we see that $c_{n_1 \star n_2} = \varphi_d(w_{n_1}w_{n_2}\infty) =\varphi_d(w_{n_2}\infty) + c_{n_1} = c_{n_1} + c_{n_2}$, so $c$ is a homomorphism.

Since $W(d)$ is an elementary abelian $2$-group, its image is killed
by $2$.  Finally all cusps of $X_0(d)$ are rational for $d$ squarefree, so
$(w_n\infty) - (\infty)$ is a rational divisor class and $c_n \in E_f(\Q)$.
\end{proof}

\begin{proposition}\label{prop: elliptic cover info and degree}
Let $N$ be squarefree, let $f$ be a newform of level $d \mid N$ with rational
coefficients and all Atkin--Lehner eigenvalues $+1$, and let $\varphi_d$, $c$ be
as in Lemma~\ref{lem: obstruction}.  Put $C \colonequals \im(c) \subseteq
E_f(\Q)[2]$, let $\lambda \colon E_f \to E^C_f \colonequals E_f/C$ be the
quotient isogeny, and set
\[
  \delta_f \colonequals \lvert C \rvert \cdot \deg(\varphi_d)/2^{\omega(d)} .
\]
Then there is a morphism $\pi_f \colon X_0^*(N) \to E^C_f$ of degree
\begin{equation}\label{eq: cover degree}
  \deg(\pi_f) = \delta_f \prod_{\ell \mid N/d}\bigl( \ell + 1 + a_\ell(f) \bigr).
\end{equation}
Taking $N = d$ shows $\delta_f = \deg\bigl( X_0^*(d) \to E^C_f \bigr)$; in
particular $\delta_f$ is a positive integer.  If moreover $d \neq N$, then
$\pi_f = \lambda \circ \pi_f'$ for a morphism $\pi_f' \colon X_0^*(N) \to E_f$ of
degree $\deg(\pi_f)/\lvert C\rvert$.
\end{proposition}
 
\begin{proof} We begin with some preliminaries about degeneracy operators and Hecke operators.
For $m \mid N/d$ let $i_m \colon X_0(N) \to X_0(d)$ be the degeneracy map
induced by $z \mapsto mz$.  If $x,y,z,u \in \Z $ satisfy $m^2xu - Nyz = m$, then $w_m$
on $X_0(N)$ is represented by
\[
\begin{pmatrix} mx & y \\ Nz & mu \end{pmatrix}
=
\begin{pmatrix} x & y \\ (N/m)z & mu \end{pmatrix}
\begin{pmatrix} m & 0 \\ 0 & 1 \end{pmatrix}
\in
\Gamma_0(d) \begin{pmatrix} m & 0 \\ 0 & 1 \end{pmatrix},
\]
the first factor lying in $\Gamma_0(d)$ because $d \mid N/m$, with determinant
$(m^2xu - Nyz)/m = 1$.  Hence
\begin{equation}\label{eq: degeneracy is AL twist}
  i_m = i_1 \circ w_m .
\end{equation}
Set $\pi_m \colonequals \varphi_d \circ i_m$ and $\pi \colonequals \sum_{m \mid
N/d} \pi_m \colon X_0(N) \to E_f$, the summation being addition on $E_f$.  We set
$t \colonequals 2^{\omega(N/d)}$ to be the number of divisors of $N/d$.
 
For $m \mid N/d$,
\begin{equation}\label{eq: degeneracy Hecke}
  i_{m,*}\, i_1^{*} = \Bigl( \prod_{\ell \mid N/(dm)} (\ell+1) \Bigr) \cdot T_m
  \qquad \text{on } J_0(d).
\end{equation}
Indeed, on moduli $i_1^{*}(E, C_d) = \sum_{C_{N/d}} (E, C_d, C_{N/d})$, the sum
running over the $\prod_{\ell \mid N/d}(\ell+1)$ cyclic subgroups of order $N/d$
that meet $C_d$ trivially, while $i_m(E, C_d, C_{N/d}) = (E/C_m, \overline{C_d})$
depends only on the subgroup $C_m \subseteq C_{N/d}$ of order $m$.  Each such
$C_m$ arises from exactly $\prod_{\ell \mid N/(dm)}(\ell+1)$ subgroups
$C_{N/d}$, and summing over $C_m$ gives $T_m$.
 
We now compute the degree of $\pi$.
For a nonconstant morphism $\phi
\colon X_0(N) \to E_f$ write $\phi^{*} \colon E_f \to J_0(N)$ for pullback of
divisor classes and $\phi_{*} \colon J_0(N) \to E_f$ for the norm, so
$\phi_{*}\phi^{*} = \deg \phi$ in $\End(E_f)$.
Both $\phi \mapsto \phi_{*}$
and $\phi \mapsto \phi^{*}$ are additive in $\phi$: for $\phi_{*}$ this is
immediate from the Albanese description on divisors, and for $\phi^{*}$ it
follows from $\sigma^{*}L \cong p_1^{*}L \otimes p_2^{*}L$ for $L \in
\Pic^0(E_f)$, with $\sigma, p_1, p_2 \colon E_f \times E_f \to E_f$ the sum and
the projections \cite[\S 5, Cor.\ 6]{Mumford}. 
Hence
$\pi_{*}\pi^{*} = \sum_{m_1,m_2} \pi_{m_1,*} \pi_{m_2}^{*}$.  Fix $m_1,m_2$ and
set $m \colonequals m_1 \star m_2$.  Since $w_{m_2}$ is an involution,
$w_{m_2}^{*} = w_{m_2,*}$, so by \eqref{eq: degeneracy is AL twist},
\[
  \pi_{m_1,*} \pi_{m_2}^{*}
  = \varphi_{d,*}\, i_{1,*}\, w_{m_1,*} w_{m_2,*}\, i_1^{*}\, \varphi_{d}^{*}
  = \varphi_{d,*}\, i_{m,*}\, i_1^{*}\, \varphi_{d}^{*},
\]
using $i_{1,*}w_{m,*} = i_{m,*}$.  Combining \eqref{eq: degeneracy Hecke} with
$\varphi_{d,*} \circ T_m = a_m(f) \varphi_{d,*}$ and $\varphi_{d,*}
\varphi_{d}^{*} = \deg(\varphi_d)$, and noting that each $m \mid N/d$ arises as
$m_1 \star m_2$ from exactly $t$ pairs,
\[
  \deg \pi
  = t \deg(\varphi_d) \sum_{m \mid N/d} a_m(f) \prod_{\ell \mid N/(dm)} (\ell+1)
  = 2^{\omega(N)} \frac{\deg(\varphi_d)}{2^{\omega(d)}}
    \prod_{\ell \mid N/d} \bigl( \ell+1+a_\ell(f) \bigr),
\]
the second equality on expanding the product, each $\ell$ contributing $\ell+1$
if $\ell \nmid m$ and $a_\ell(f)$ if $\ell \mid m$, and using $a_m(f) =
\prod_{\ell \mid m} a_\ell(f)$ for $m$ squarefree and coprime to $d$.  By the
Weil bound $\lvert a_\ell(f)\rvert \le 2\sqrt\ell < \ell+1$, every factor is
positive, so $\pi_{*}\pi^{*} \neq 0$; hence $\pi$ is nonconstant and
$\deg\pi = \pi_{*}\pi^{*}$.  (For $d = N$ the product is empty and this reads
$\deg(\varphi_N) = 2^{\omega(N)}\delta_f/\lvert C\rvert$.)

Let $n \mid N$
and write $n = n_1n_2$ with $n_1 \mid d$, $n_2 \mid N/d$; this is
unique since $N$ is squarefree.  We now show that $\pi \circ w_n =  \pi + t \cdot c_{n_1}$, which will allow us to descend computations to the quotient.
By \eqref{eq: degeneracy is AL twist} and the
commutativity of $W(N)$,
\[
  \pi_m \circ w_n = \varphi_d \circ i_1 \circ w_m w_{n_1} w_{n_2}
  = \varphi_d \circ i_1 \circ w_{n_1} \circ w_{m \star n_2} .
\]
Since $n_1 \mid d$, the degeneracy map $i_1$ intertwines the Atkin--Lehner
involutions at $n_1$ on the two levels, $i_1 \circ w^{(N)}_{n_1} =
w^{(d)}_{n_1} \circ i_1$,
so Lemma~\ref{lem: obstruction} gives $\pi_m \circ w_n
= \pi_{m \star n_2} + c_{n_1}$.  Summing over $m \mid N/d$, during which
$m \mapsto m \star n_2$ merely permutes the summands, we see,
\begin{equation}\label{eq: invariance defect}
  \pi \circ w_n = \pi + t \cdot c_{n_1}.
\end{equation}
 
Finally we descend to the quotient.
Since $c_{n_1} \in C = \ker\lambda$, equation
\eqref{eq: invariance defect} gives $(\lambda \circ \pi) \circ w_n = \lambda
\circ \pi$ for every $n \mid N$, so $\lambda \circ \pi$ factors as $\pi_f \circ
q_N$ with $q_N \colon X_0(N) \to X_0^*(N)$ of degree $2^{\omega(N)}$.  Dividing
$\deg(\lambda \circ \pi) = \lvert C\rvert \deg \pi$ by $2^{\omega(N)}$ gives
\eqref{eq: cover degree}.  If $d \neq N$ then $t$ is even while $c_{n_1} \in
E_f[2]$, so $t \cdot c_{n_1} = 0$ and $\pi$ is itself $W(N)$-invariant; it
descends to $\pi_f' \colon X_0^*(N) \to E_f$ with $\pi_f = \lambda \circ \pi_f'$. 
\end{proof}
 
\begin{remark}\label{rem: delta}
The naive normalization $\deg(\varphi_d)/2^{\omega(d)}$ is \emph{not} in general
the degree of a morphism out of $X_0^*(d)$, and need not be an integer: by
Lemma~\ref{lem: obstruction}, $\varphi_d$ descends to $X_0^*(d)$ only when
$c = 0$, which holds whenever $E_f(\Q)[2] = 0$ but can fail.  The corrected
constant $\delta_f$ absorbs the factor $\lvert C\rvert$ and is a genuine degree,
with target $E^C_f$ rather than $E_f$.
\end{remark}
 
\begin{example}\label{ex: level 65}
For $d = N = 65$ one has $E_f = \ecl{65.a.1}$, $\deg(\varphi_{65}) = 2$ and
$\lvert C\rvert = 2$, so $\delta_f = 2 \cdot 2/4 = 1$ and $X_0^*(65) \cong
E^C_f = \ecl{65.a.2}$, the curve $2$-isogenous to $E_f$, in agreement with
the standard model of this quotient.  Note that $\deg(\varphi_{65})/2^{\omega(65)}
= 1/2$ is not an integer, and that the target is not the $X_0(65)$-optimal curve.
The same phenomenon occurs at $d = N = 185$, where $E_f = \ecl{185.c.1}$ and
$E^C_f = \ecl{185.c.2}$.  These are the only two levels $d$ with $c \neq 0$
among the divisors of the levels $N$ occurring in
Table~\ref{tab: triple covers} and in
Tables~\ref{tab:exceptional_points_genus_3}
and~\ref{tab:exceptional_points_genus_4}.
\end{example}

 Finally, the next lemma addresses the minimality of $\pi_f$.
\begin{lemma}\label{lem: manin}
In the situation of Proposition~\ref{prop: elliptic cover info and degree},
suppose $c = 0$, so that $E^C_f = E_f$.  Let $\rho \colon X_0^*(N) \to E_f$
be a morphism sending the cusp to $O_{E_f}$ and with $\rho_*$ generating the rank 1
group $\Hom(J_0^*(N), E_f)$, and write $\pi_f = [k] \circ \rho$.  Then $k$
divides the Manin constant $c_d$ of $\varphi_d$.  In particular, if $c_d = 1$
then $\deg(\pi_f)$ is the least degree of a morphism $X_0^*(N) \to E_f$.
\end{lemma}
 
\begin{proof}
Let $\mu \colon X_0(N) \to E_f$ be the optimal parametrization, so that $\mu_*$
generates the space $\Hom(J_0(N), E_f)$, and write $\pi = [a] \circ \mu$ and
$\rho \circ q_N = [b] \circ \mu$ with $a,b \in \Z$.  From $\pi = \pi_f \circ q_N
= [k] \circ \rho \circ q_N$ we get $a = kb$, so it suffices to show
$a \mid c_d$.
 
Let $\omega$ be the N\'eron differential of $E_f$.  
Pullback of invariant differentials is additive in the morphism, (see Step~3 Proposition~\ref{prop: elliptic cover info and degree}), so 
\[
  \pi^{*}\omega = \sum_{m \mid N/d} i_m^{*}\varphi_d^{*}\omega
  = c_d \cdot 2\pi i \Bigl( \sum_{m \mid N/d} m f(mz) \Bigr) dz ,
\]
using $\varphi_d^{*}\omega = c_d \cdot 2\pi i f(z)\,dz$ and $i_m^{*}(f(z)dz) =
m f(mz)\,dz$. 
The form $h(z) \coloneqq \sum_{m \mid N/d} m f(mz)$ in $S_2(\Gamma_0(N))$, has integral
$q$-expansion and has $a_1(h) = 1$ since only $m = 1$ contributes.  On the other
hand $\pi^{*}\omega = a\,\mu^{*}\omega$, and $\mu^{*}\omega$ extends to a global
section of the sheaf of regular differentials on the minimal regular model of $X_0(N)$ over $\Z$, hence has
integral $q$-expansion at $\infty$ by the $q$-expansion principle for $X_0(N)$ over $\Z$
\cite{DeligneRapoport}.  Writing $\mu^{*}\omega = 2\pi i\,G(z)\,dz$ we get
$G = (c_d/a)h$ with $G$ integral, and comparing first coefficients gives
$a \mid c_d$.
\end{proof}
 
\begin{remark}\label{rem: manin one}
The Manin constant of every optimal quotient occurring in
Proposition~\ref{prop: triple cover classification} equals $1$, as recorded in
\cite{CremonaTables,lmfdb};
Lemma~\ref{lem: manin}, the morphism $\pi_f$ of
Proposition~\ref{prop: elliptic cover info and degree} is therefore minimal in
all cases relevant below, and \eqref{eq: cover degree} computes the least degree
of an elliptic cover with the given target.  (When $\deg(\pi_f)$ is squarefree
this is automatic, since $\deg(\pi_f) = k^2\deg(\rho)$.)
\end{remark}

     When $\deg(\pi_f) = 2$, the curve $X_0^*(N)$ is bielliptic and an exceptional point is explained by elliptic double cover if and only if it is explained by the corresponding bielliptic involution.  This applies to all but two exceptional points in genus $3$ (see Table \ref{tab:exceptional_points_genus_3}), and demonstrates yet another overlap in our  constructions.

    To get previously unexplained points, we consider covers of degree $3$ or higher.

\begin{example}\label{ex: N = Mell}
    Suppose $N = M\ell$ for a prime $\ell \nmid M$ and take $d = M$.  Assume
    that $X_0^*(M)$ has genus $1$; taking the cusp as origin makes it an
    elliptic curve, and Proposition~\ref{prop: elliptic cover info and degree}
    applied with $N$ replaced by $M$ gives an isogeny
    $X_0^*(M) \to E^C_f$ of degree $\delta_f$, which is an isomorphism when
    $\delta_f = 1$.  Composing the map
    $\Pi \colon X_0^*(M\ell) \to \Sym^2(X_0^*(M))$ sending a $W(M\ell)$-orbit to
    the two $W(M)$-orbits comprising it with the addition law
    $\Sym^2(X_0^*(M)) \to X_0^*(M)$ then yields, under this identification, the
    morphism $\pi_f$ of Proposition~\ref{prop: elliptic cover info and degree}.
    Note that the target is $E^C_f$ and not $E_f$: for $M = 65$ one has
    $X_0^*(65) \cong \ecl{65.a.2}$ while $E_f = \ecl{65.a.1}$
    (Example~\ref{ex: level 65}), and the two differ by the $2$-isogeny
    $\lambda$.
    From the moduli interpretation, it is easy to see that $\Pi$ is injective on non-CM points, hence generically injective. However, it may happen that $\Pi(x_1) = \Pi(x_2)$ for CM points $x_1,x_2$, as the next example shows.
\end{example}

\begin{example}\label{ex: 286}
     We specialize Example \ref{ex: N = Mell} to $N = 286 = 2 \cdot 11 \cdot 13$ and $\ell = 2$. The curve $\href{https://www.lmfdb.org/ModularForm/GL2/Q/holomorphic/143/2/a/a/}{E_f = X_0^*(143)}$ has $a_2(f) = 0$, so $\pi_f \colon X_0^*(286) \to X_0^*(143)$ is a triple cover. 
     Let $K = \Q(\sqrt{-39})$. 
     By Example \ref{ex: N = 286 D = -39}, 
     there are two points $x_1,x_2 \in X_0^*(286)(\overline \Q)$ with CM by $\calO_K$,
represented by the two ``isogeny cubes'' in Figure \ref{fig:my_image}. 
The shaded/striped symmetry shows that the map $\Pi \colon X_0^*(286) \to \Sym^2(X_0^*(143))$ from Example \ref{ex: N = Mell} sends $x_1$ and $x_2$ to the same unordered pair of ``isogeny squares''. The corresponding fiber of $\pi_f$ is therefore $\{x, x_1, x_2\}$ for some $x \in X_0^*(286)(\overline\Q)$. Since $\{x_1,x_2\} = \overline{\mathrm{CM}}(\calO_K)$ is $\Gal_\Q$-stable, so is the fiber, and hence $x$ is a rational point. We check in Magma that $x$ is one of the exceptional points on $X_0^*(286)$.
\begin{figure}[htbp]
    \centering 
\begin{tikzpicture}[line cap=round, line join=round,
    edge/.style={black, line width=0.9pt},
    behind/.style={black!50, line width=0.6pt, dash pattern=on 2.2pt off 2.4pt},
    E0/.style={circle, fill=black, minimum size=6.8pt, inner sep=0pt},
    E1/.style={circle, draw=black, fill=white, line width=0.7pt, minimum size=6.8pt, inner sep=0pt},
    E2/.style={rectangle, fill=black, minimum size=6.4pt, inner sep=0pt},
    E3/.style={rectangle, draw=black, fill=white, line width=0.7pt, minimum size=6.4pt, inner sep=0pt},
    elab/.style={font=\footnotesize, inner sep=1pt},
    leg/.style={font=\footnotesize, inner sep=1.5pt},
    facegray/.style={fill=black!10},
    facehatch/.style={pattern=north east lines, pattern color=black!40}]

\begin{scope}
\coordinate (v000) at (0,0);
\coordinate (v100) at (2.6,0);
\coordinate (v010) at (1.3,0.9);
\coordinate (v110) at (3.9,0.9);
\coordinate (v001) at (0,2.6);
\coordinate (v101) at (2.6,2.6);
\coordinate (v011) at (1.3,3.5);
\coordinate (v111) at (3.9,3.5);

\fill[facehatch] (v000) -- (v100) -- (v110) -- (v010) -- cycle;
\fill[facegray]  (v001) -- (v101) -- (v111) -- (v011) -- cycle;

\draw[behind] (v000) -- (v010);
\draw[behind] (v010) -- (v110);
\draw[behind] (v010) -- (v011);
\draw[edge] (v000) -- (v100) node[elab, midway, below=2.5pt] {$13$};
\draw[edge] (v100) -- (v110) node[elab, midway, below right=1.5pt] {$11$};
\draw[edge] (v000) -- (v001) node[elab, midway, left=2.5pt] {$2$};
\draw[edge] (v100) -- (v101);
\draw[edge] (v110) -- (v111);
\draw[edge] (v001) -- (v101);
\draw[edge] (v101) -- (v111);
\draw[edge] (v011) -- (v111);
\draw[edge] (v001) -- (v011);

\node[E0] at (v000) {};
\node[E2] at (v100) {};
\node[E1] at (v010) {};
\node[E3] at (v110) {};
\node[E1] at (v001) {};
\node[E3] at (v101) {};
\node[E2] at (v011) {};
\node[E0] at (v111) {};
\end{scope}

\begin{scope}[xshift=7.4cm]
\coordinate (v000) at (0,0);
\coordinate (v100) at (2.6,0);
\coordinate (v010) at (1.3,0.9);
\coordinate (v110) at (3.9,0.9);
\coordinate (v001) at (0,2.6);
\coordinate (v101) at (2.6,2.6);
\coordinate (v011) at (1.3,3.5);
\coordinate (v111) at (3.9,3.5);

\fill[facegray]  (v000) -- (v100) -- (v110) -- (v010) -- cycle;
\fill[facehatch] (v001) -- (v101) -- (v111) -- (v011) -- cycle;

\draw[behind] (v000) -- (v010);
\draw[behind] (v010) -- (v110);
\draw[behind] (v010) -- (v011);
\draw[edge] (v000) -- (v100) node[elab, midway, below=2.5pt] {$13$};
\draw[edge] (v100) -- (v110) node[elab, midway, below right=1.5pt] {$11$};
\draw[edge] (v000) -- (v001) node[elab, midway, left=2.5pt] {$2$};
\draw[edge] (v100) -- (v101);
\draw[edge] (v110) -- (v111);
\draw[edge] (v001) -- (v101);
\draw[edge] (v101) -- (v111);
\draw[edge] (v011) -- (v111);
\draw[edge] (v001) -- (v011);

\node[E0] at (v000) {};
\node[E2] at (v100) {};
\node[E3] at (v010) {};
\node[E1] at (v110) {};
\node[E1] at (v001) {};
\node[E3] at (v101) {};
\node[E0] at (v011) {};
\node[E2] at (v111) {};
\end{scope}

\end{tikzpicture}

    \caption{CM points of discriminant $-39$ on $X_0^*(286)$. Each elliptic curve $E_0, \dots, E_3$ is represented by a shape (filled/empty circle, filled/empty square) while $2$,$11$, and $13$-isogenies connect along each coordinate axis.}
    \label{fig:my_image}
\end{figure}
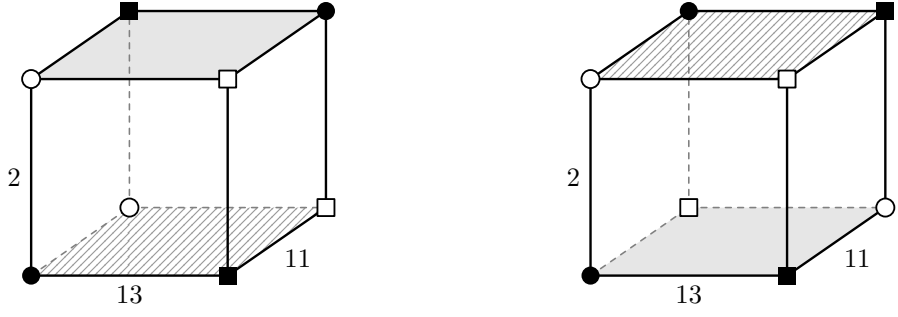
\end{example}

Example \ref{ex: 286} shows that  ``elliptic cover'' explanations can sometimes be seen directly from the  moduli interpretation of the cover and the symmetries of  CM points (though one must still check that the ``explained'' point is exceptional and not some other CM point).  For the more general triple covers appearing in Proposition \ref{prop: elliptic cover info and degree}, we use the method of \cite{BarsGonzalezRovira} to compute explicit equations for the map in Magma, and we check if any rational or quadratic CM points happen to share a fiber. We find $11$ total instances of this; see Table~\ref{tab:exceptional_cm_fibers}. Most, but not all, of these exceptional points have already been explained in other ways. 

  \begin{table}[ht!]
  \centering
  \begin{tabular}{ccccl}
  \toprule
  $N$ & \textbf{Genus}   & $E$ & $D$ &  \textbf{Fiber type} \\
  \midrule
  $154$   & $2$ & \ecl{77.a.1}  & $-63$        & degree $2$ CM point \\
  $154$ & 2  & \ecl{154.a.2} & $-28$        & ramified rational CM \\
  $285$ & $2$ & \ecl{57.a.1}  & $-219$   & degree $2$ CM point \\
  $285$ & $2$& \ecl{285.b.2} & $-15$      & ramified rational CM \\
  $286$ & $2$  & \ecl{143.a.1} & $-39$    & degree $2$ CM point \\
  $286$ & $2$  & \ecl{286.b.1} & $-220$    & degree $2$ CM point \\
  $246$ & $3$  & \ecl{123.b.1} & $-228$      & ramified rational CM \\
  $310$ & $3$  & \ecl{155.c.1} & $-520$         & ramified rational CM \\
  $318$ & $3$  & \ecl{318.a.1} & $-15$         & ramified rational CM \\
  $430$ & $3$  & \ecl{430.a.1} & $-280$      & ramified rational CM \\
  $399$ & $4$  & \ecl{57.a.1}  & $-75,\ -483$  & split rational CM \\
  \bottomrule
  \end{tabular}
  \caption{Exceptional points in special fibers of triple covers
  $X_0^*(N)\to E$, where $E$ is given by a Cremona label. }
  \label{tab:exceptional_cm_fibers}
  \end{table}

Example \ref{ex: 286} and Table \ref{tab:exceptional_cm_fibers} show that the two remaining exceptional points on $X_0^*(286)$ are explained by elliptic triple cover. We therefore deduce that (to a high degree of precision)
    all $65$ known exceptional points on star curves $X_0^*(N)$ of squarefree level are explained by (at least one of) automorphism, collinearity, or elliptic triple cover.

We also computed  all elliptic triple covers, to see if any low degree CM points yield new exceptional points (as unlikely as it is that these would not have been found in our broad search). Since $\gamma(X_0^*(N)) \to \infty$ as $N \to \infty$, there are only finitely many star curves admitting an elliptic triple cover. More precisely, we have the following classification. 

\begin{proposition}\label{prop: triple cover classification}
Let $N$ be squarefree.  The curve $X_0^*(N)$ admits a triple cover of an elliptic
curve $E$ if and only if the triple $(N,d,E)$, with $d$ the level of the
corresponding newform, occurs in Table~\ref{tab: triple covers}.
\end{proposition}
 
\begin{proof}
We follow the basic strategy of \cite{BarsGonzalezRovira}.  Let $\psi \colon X_0^*(N) \to E$ be a triple cover, and let $p$ be the smallest
prime not dividing $N$.  Every prime $q < p$ divides $N$, so
$\prod_{q<p}(q+1)/2 \le \prod_{q \mid N} (q+1)/2$, and combining with
Lemma~\ref{lem:supersingularbd} below,
\begin{equation}\label{eq:inequality_fund}
  \prod_{q<p} \frac{q+1}{2} \le 12 \cdot \frac{3(p+1)^2 - 1}{p-1}.
\end{equation}
Write $L(p)$ and $R(p)$ for the left and right sides respectively.  One has $L(17) = 1512 > 728.25 =
R(17)$.  If $p \ge 17$ and $r$ is the next prime, then $r < 2p$ by Bertrand's
postulate, $L(r) = \tfrac{p+1}{2}L(p) \ge 9L(p)$, and $R(r) < R(2p) < 2R(p)$;
hence $L(p) > R(p)$ implies $L(r) > R(r)$, and by induction
\eqref{eq:inequality_fund} fails for every $p \ge 17$.  So the smallest prime not
dividing $N$ is less than $17$, and for each such $p$
Lemma~\ref{lem:supersingularbd} bounds $\prod_{q \mid N}(q+1)/2$, hence the size
and number of the prime divisors of $N$.  Only finitely many $N$ occur and they
can be enumerated.
 
The cover $\psi$ factors through a surjection $J_0^*(N) \twoheadrightarrow E$, so
$E$ is an isogeny factor of $J_0^*(N)$; thus $E$ is isogenous to $E_f$ for a
newform $f$ of some level $d \mid N$ whose Atkin--Lehner eigenvalues at the
primes dividing $d$ are all $+1$. 
The $f$-old space at level $N$ has basis $\{f(mz)\}_{m \mid N/d}$.  For
$\ell \mid d$, each basis vector is an eigenform of $w_\ell$ with eigenvalue
$\varepsilon_\ell(f) = +1$ \cite[(5.2)]{AtkinLehner}, so $w_\ell$ acts
trivially.  
For $\ell \mid N/d$,  the combinations $f(mz)\pm \ell f(\ell mz)$ are $w_\ell$-eigenforms with eigenvalues $\pm 1$ \cite[(5.1)]{AtkinLehner}.
Since the $w_\ell$ commute
and are simultaneously diagonalizable, their joint $+1$-eigenspace is
1-dimensional, spanned by $\sum_{m \mid N/d} m f(mz)$.  
Hence the $f$-isotypic
part of $J_0^*(N)$ is 1-dimensional and $\Hom(J_0^*(N),E)$ is free of rank
1. Writing $\rho$ for a generator we have $\psi = [n] \circ \rho$ up to translation, and
$3 = \deg \psi = n^2 \deg \rho$ forces $n = \pm 1$ and $\deg \rho = 3$.  
Since $g(X_0^*(N)) \ge 2$ for every $N$ under consideration, no morphism
$X_0^*(N) \to E$ has degree $1$. Hence $\psi$ does not factor as an isomorphism
followed by an isogeny, and $E$ is the target $E^C_f$ of
Proposition~\ref{prop: elliptic cover info and degree}.
Writing $\pi_f = [k] \circ \rho$, equation \eqref{eq: cover degree} gives
\[
  \delta_f \prod_{\ell \mid N/d} \bigl( \ell + 1 + a_\ell(f) \bigr)
  = \deg(\pi_f) = k^2 \deg(\rho) = 3k^2 .
\]
If $c= 0$ then $k = 1$ by Remark~\ref{rem: manin one}.  If
$c \ne 0$ then $E_f(\Q)[2] \neq 0$ by Lemma~\ref{lem: obstruction};
evaluating the left-hand side for every $d \mid N$ and every admissible $f$
over the finitely many $N$, one checks that
it is of the form $3k^2$ with $k > 1$ only for newforms with $c = 0$.  In
either case $k = 1$, and therefore
\[
  \delta_f \prod_{\ell \mid N/d} \bigl( \ell + 1 + a_\ell(f) \bigr) = 3 .
\]
Since $\delta_f$ is a positive integer and each factor is a positive integer, either
$\delta_f = 3$ and every $\ell \mid N/d$ has $\ell + 1 + a_\ell(f) = 1$, or
$\delta_f = 1$ and exactly one $\ell \mid N/d$ has $\ell+1+a_\ell(f) = 3$, the
others contributing $1$.  By the Weil bound, $\ell+1+a_\ell(f) = 3$ forces
$\ell \in \{2,3,5,7\}$ and $\ell+1+a_\ell(f) = 1$ forces $\ell \in \{2,3\}$.
 
If $\delta_f = 1$ then $X_0^*(d) \cong E^C_f$ has genus $1$, and the triple cover
is the map $X_0^*(N) \to X_0^*(d)$ of \eqref{eq: cover degree}, possibly after
adjoining the finitely many primes $\ell \in \{2,3\}$ contributing a factor $1$.
If $\delta_f = 3$ then either $d = N$, or $d < N$ and every $\ell \mid N/d$
satisfies $\ell + 1 + a_\ell(f) = 1$; an explicit computation shows the only such
examples are $X_0^*(402) \to E^C_{\ecl{201.a.1}}$ and $X_0^*(438) \to
E^C_{\ecl{219.a.1}}$, both arising from $\ell = 2$ with $a_2(f) = -2$.
Enumerating the surviving $N$ gives Table~\ref{tab: triple covers}.
\end{proof}
 
\begin{table}[ht!]
\centering
\small
\begin{tabular}{lll@{\qquad}lll@{\qquad}lll}
\toprule
$N$ & $d$ & $E$ & $N$ & $d$ & $E$ & $N$ & $d$ & $E$\\
\midrule
$154$ & $77$  & \ecl{77.a.1}  & $274$ & $274$ & \ecl{274.b.1} & $429$ & $143$ & \ecl{143.a.1}\\
$154$ & $154$ & \ecl{154.a.2} & $282$ & $141$ & \ecl{141.d.1} & $430$ & $430$ & \ecl{430.a.1}\\
$163$ & $163$ & \ecl{163.a.1} & $285$ & $57$  & \ecl{57.a.1}  & $434$ & $434$ & \ecl{434.a.2}\\
$185$ & $185$ & \ecl{185.c.2} & $285$ & $285$ & \ecl{285.b.2} & $438$ & $219$ & \ecl{219.a.1}\\
$201$ & $201$ & \ecl{201.a.1} & $286$ & $143$ & \ecl{143.a.1} & $455$ & $91$  & \ecl{91.a.1}\\
$202$ & $101$ & \ecl{101.a.1} & $286$ & $286$ & \ecl{286.b.1} & $462$ & $77$  & \ecl{77.a.1}\\
$214$ & $214$ & \ecl{214.a.1} & $290$ & $58$  & \ecl{58.a.1}  & $465$ & $155$ & \ecl{155.c.1}\\
$219$ & $219$ & \ecl{219.a.1} & $291$ & $291$ & \ecl{291.c.2} & $570$ & $57$  & \ecl{57.a.1}\\
$237$ & $79$  & \ecl{79.a.1}  & $305$ & $61$  & \ecl{61.a.1}  & $570$ & $190$ & \ecl{190.a.1}\\
$246$ & $123$ & \ecl{123.b.1} & $310$ & $155$ & \ecl{155.c.1} & $574$ & $574$ & \ecl{574.b.1}\\
$249$ & $83$  & \ecl{83.a.1}  & $318$ & $318$ & \ecl{318.a.1} & $590$ & $118$ & \ecl{118.a.1}\\
$254$ & $254$ & \ecl{254.a.1} & $354$ & $118$ & \ecl{118.a.1} & $798$ & $57$  & \ecl{57.a.1}\\
$258$ & $258$ & \ecl{258.b.1} & $393$ & $131$ & \ecl{131.a.1} & $870$ & $58$  & \ecl{58.a.1}\\
$262$ & $131$ & \ecl{131.a.1} & $395$ & $79$  & \ecl{79.a.1}  & $910$ & $91$  & \ecl{91.a.1}\\
$262$ & $262$ & \ecl{262.a.1} & $399$ & $57$  & \ecl{57.a.1}  &       &       & \\
$267$ & $89$  & \ecl{89.a.1}  & $402$ & $201$ & \ecl{201.a.1} &       &       & \\
$269$ & $269$ & \ecl{269.a.1} & $426$ & $142$ & \ecl{142.a.1} &       &       & \\
\bottomrule
\end{tabular}
\caption{Triples $(N,d,E)$ for which $X_0^*(N)$ admits a triple cover of the
elliptic curve $E$, with $d$ the level of the associated newform.  The label is
that of $E = E^C_f$, which differs from $E_f$ exactly when $c \ne 0$.
This occurs only at $N = d = 185$, where $E_f = \ecl{185.c.1}$.}
\label{tab: triple covers}
\end{table}

\begin{lemma}
\label{lem:supersingularbd}

Let $N$ be squarefree and suppose $\psi:X_0^*(N) \to E$ admits a degree 3 morphism over $\Q$ with $E$ an elliptic curve.  
Then for every prime $p \nmid N$, we have
\[ \frac{\Psi(N)}{2^{\omega(N)}} = \prod_{q\mid N}\frac{q+1}{2} \leq 12\cdot \frac{3(p+1)^2 - 1}{p - 1}\]
where $\Psi$ is the Dedekind psi function $\Psi(N) = \prod_{q\mid N}(q+1)$ for squarefree $N$.
\end{lemma}

\begin{proof}
The composite morphism $X_0(N) \to E$ has degree $3 \cdot 2^{\omega(N)}$.
Then $X_0(N)(\F_{p^2})$ has $2^{\omega(N)}$ rational cusps and at least $(p-1)\Psi(N)/12$ supersingular points \cite[Lemmas 3.20 and 3.21]{BGGP}.
Hence \[ 2^{\omega(N)} + (p-1)\Psi(N)/12 \leq 3 \cdot 2^{\omega(N)} \#E(\F_{p^2}) \leq 3  \cdot 2^{\omega(N)}  (p+1)^2\]
where we apply the bound $\# E(\F_{p^2}) \leq (p+1)^2$ to obtain the second inequality. Rearranging gives the lemma. 
\end{proof}

Table \ref{tab:exceptional_cm_fibers} was constructed by running our algorithms on the pairs $(N,E)$ that appear in Proposition \ref{prop: triple cover classification} and all fibers containing a CM point or cusp. Thus, there are provably no more exceptional points that are explained by triple cover. 
This gives additional evidence for Conjecture \ref{conj:explicitbd}.

We can also ask about exceptional points arising from higher degree covers.  For
the elliptic cover construction to work in degree $\delta$, one needs
$\delta - 2$ ``coincidences'', since $\delta-2$ further points in the fibre of a
special point must themselves be special; so for $\delta \ge 4$ we expect this to
be exceedingly rare, and even when the construction produces a rational point it
is not forced to be exceptional.  We found no such exceptional points in degree $\delta \geq 4$.  It is worth noting that
apparent instances can arise spuriously: at $N = 455$ the quotient $X_0^*(65) =
\ecl{65.a.2}$ is the target of a degree-$4$ map from $X_0^*(455)$, but by
Proposition~\ref{prop: elliptic cover info and degree} that map is the composite
of a \emph{double} cover $X_0^*(455) \to \ecl{65.a.1}$ with the $2$-isogeny
$\lambda$, and the corresponding explanation is the bielliptic one already
recorded in Table~\ref{tab:exceptional_points_genus_3}.

\bibliographystyle{amsalpha}
\bibliography{bibliography}

\end{document}